\documentclass[11pt]{article}
\usepackage{amssymb,amsmath,amsthm}
\usepackage{color}
\usepackage{marginnote}

\numberwithin{equation}{section}
\newtheorem{thm}{Theorem}[section]

\newtheorem{prop}[thm]{Proposition}
\newtheorem{lem}[thm]{Lemma}

\theoremstyle{definition}
\newtheorem{rem}[thm]{Remark}
\newtheorem{rems}[thm]{Remarks}
\let\oldproofname=\proofname
\renewcommand{\proofname}{\rm\bf{\oldproofname}}

\newcommand{\N}{\mathbb{N}}
\newcommand{\Z}{\mathbb{Z}}

\newcommand{\R}{\mathbb{R}}
\newcommand{\C}{\mathbb{C}}

\newcommand{\cL}{\mathcal{L}}
\newcommand{\cM}{\mathcal{M}}
\newcommand{\cN}{\mathcal{N}}
\newcommand{\cO}{\mathcal{O}}

\renewcommand{\Re}{\mathop{\mathrm{Re}}}
\renewcommand{\Im}{\mathop{\mathrm{Im}}}
\newcommand{\dd}{\,{\rm d}}
\newcommand{\D}{{\rm d}}
\renewcommand{\:}{\thinspace :}
\renewcommand{\div}{\mathop{\mathrm{div}}\nolimits}

\newcommand{\curl}{\mathop{\mathrm{curl}}}

\newcommand{\dist}{\mathop{\mathrm{dist}}}

\newcommand{\QED}{\mbox{}\hfill$\Box$}
\newcommand{\Lamad}{\Lambda_{\rm ad}}
\newcommand{\Lamnl}{\Lambda_{\rm nl}}

\newdimen\texpscorrection
\newdimen\figcenter
\def\figurewithtex #1 #2 #3 #4 #5\cr{\null
  {\goodbreak\figcenter=\hsize\relax
  \advance\figcenter by -#4truecm
  \divide\figcenter by 2
  \begin{figure}[hbt]
  \vskip #3truecm\noindent\hskip\figcenter
  \includegraphics{#1}{\hskip\texpscorrection\input #2 }
  \vskip 0.8truecm{\baselineskip=0.8\baselineskip
  \noindent \vbox{\noindent {\footnotesize #5}}\par}
  \end{figure}}}
\def\point#1 #2 #3 {\rlap{\kern #1 truecm
\raise #2 truecm \hbox{#3}}}

\begin{document}

\title{Enhanced dissipation and axisymmetrization \\
 of two-dimensional viscous vortices}

\author{
\null\\
{\bf Thierry Gallay}\\
Institut Fourier\\
Universit\'e Grenoble Alpes et CNRS\\
100 rue des Maths, 38610 Gi\`eres, France\\
{\tt Thierry.Gallay@univ-grenoble-alpes.fr}}

\date{}

\maketitle

\begin{abstract}
This paper is devoted to the stability analysis of the Lamb-Oseen
vortex in the regime of high circulation Reynolds numbers. When 
strongly localized perturbations are applied, it is shown that 
the vortex relaxes to axisymmetry in a time proportional to 
$Re^{2/3}$, which is substantially shorter than the diffusion 
time scale given by the viscosity. This enhanced dissipation effect
is due to the differential rotation inside the vortex core. Our 
result relies on a recent work by Li, Wei, and Zhang \cite{LWZ}, 
where optimal resolvent estimates for the linearized operator 
at Oseen's vortex are established. A comparison is made with 
the predictions that can be found in the physical literature, 
and with the rigorous results that were obtained for shear 
flows using different techniques. 
\end{abstract}

\section{Introduction}\label{sec1}

It is a well known experimental fact that isolated vortices in
two-dimensional viscous flows relax to axisymmetry in a relatively
short time, because the differential rotation in the vortex core
creates small spatial scales in the radial direction which
substantially enhance the viscous dissipation \cite{Lu,RY,BL}.  This
effect can be quantified in terms of the circulation Reynolds number
$Re = |\Gamma|/\nu$, where $\Gamma$ denotes the total circulation of
the vortex and $\nu$ is the kinematic viscosity of the fluid. Whereas
axisymmetric concentrations of vorticity evolve toward the Gaussian
profile of the Lamb-Oseen vortex on the diffusion time scale, which is
$\cO(Re)$, it is observed that non-axisymmetric perturbations that
preserve the first moment of vorticity decay to zero in a much shorter
time, proportional to $Re^{1/3}$ \cite{BL,BG2}. Actually, the enhanced
dissipation effect is necessarily weaker near the vortex center where 
the differential rotation degenerates, and it can be shown that 
axisymmetrization occurs in that region on a slower time scale, 
typically of the order of $\Re^{1/2}$ \cite{BBG}. Similar conclusions
are reached when considering the distribution of a passive scalar advected 
by a vortex or by a shear flow \cite{BG2,BBG}. 

From a mathematical point of view, however, it is quite difficult 
to obtain rigorous results that describe under which assumptions
axisymmetrization occurs for two-dimensional vortices, and at which
rate. In the physical literature, the perturbations of an axisymmetric 
vortex are mostly studied at the level of the linearized equations, and 
both the time evolution of the underlying vortex and the deformation
of its streamlines are often neglected, so that vorticity is considered
as a passive scalar advected by a stationary flow. In this paper, we
focus on the archetypal example of the Lamb-Oseen vortex, and we 
establish the first stability result that describes and exploits 
the enhanced dissipation effect for the full nonlinear problem. Our
analysis shows that non-axisymmetric perturbations that preserve 
the first moment of vorticity disappear in a time proportional 
to $Re^{2/3}$, which is longer than the physical scales $Re^{1/3}$ 
and $Re^{1/2}$ obtained in \cite{Lu,BL,BG2}, but still substantially 
smaller than the diffusive scale. The origin of the new exponent 
$2/3$ will be explained in Section~\ref{subsec2.5} below, once
the resolvent estimates for the linearized operator will be presented. 

It is important to realize that axisymmetrization at high Reynolds
numbers plays a crucial role not only in the stability analysis of
isolated vortices in freely decaying turbulence, but also in a number
of related problems. For instance, Burgers vortices are stationary
solutions of the three-dimensional Navier-Stokes equations in the
presence of a linear strain field, and it is observed that the
streamlines of these vortices become more and more circular in the
limit of large Reynolds numbers, even if the strain is not
axisymmetric \cite{RS,MKO,JMV}. Rigorous results in this direction
have been obtained by C.E.~Wayne and the author \cite{GW3}, and by
Y.~Maekawa \cite{Ma2,Ma3}. Another example is the evolution of a
two-dimensional vortex in a time-dependent strain field, such as the
velocity field produced by a collection of other vortices. Accurate
asymptotic expansions \cite{TT,TK} and rigorous error estimates
\cite{Ga1} show that, for large Reynolds numbers, the vortex is nearly
axisymmetric and is deformed in such a way that the self-interaction
exactly compensates for the action of the exterior strain, except 
for a rigid translation. It should be mentioned, however, that 
the results presented below require a much deeper understanding
of the enhanced dissipation effect than what is necessary to 
prove axisymmetrization in \cite{GW3,Ga1}. The main new ingredient 
in our proof is the beautiful resolvent estimate recently obtained 
by Li, Wei, and Zhang \cite{LWZ} for the linearized operator at 
Oseen's vortex. 

We now state our results in a more precise way. Our starting point 
is the vorticity equation in the two-dimensional plane, which 
reads
\begin{equation}\label{omeq}
  \partial_t \omega(x,t) + u(x,t)\cdot\nabla \omega(x,t) \,=\, 
  \nu \Delta \omega(x,t)\,, 
\end{equation}
where $x \in \R^2$ is the space variable, $t \ge 0$ is the time
variable, and $\nu > 0$ is the kinematic viscosity. The velocity field
$u : \R^2 \times \R_+ \to \R^2$ is obtained from the vorticity 
$\omega : \R^2 \times \R_+ \to \R$ by solving the elliptic system
 $\div u = 0$, $\curl u \equiv \partial_1 u_2 - \partial_2 u_1 = 
\omega$. This leads to the two-dimensional Biot-Savart law
\begin{equation}\label{BS}
  u(x,t) \,=\, \frac{1}{2\pi}\int_{\R^2} \frac{(x-y)^\perp}{|x-y|^2}
  \,\omega(y,t) \dd y\,,
\end{equation}
which is studied in Section~\ref{subsec5.2}. In shorthand notation we 
write $u(\cdot,t) = K_{BS} * \omega(\cdot,t)$, where $K_{BS}(x) = 
\frac{1}{2\pi}\frac{x^\perp}{|x|^2}$ is the Biot-Savart kernel. 

The {\em Lamb-Oseen vortices} are self-similar solutions of 
equation \eqref{omeq} defined by
\begin{equation}\label{Osdef}
  \omega(x,t) \,=\, \frac{\Gamma}{\nu t}\,G\Bigl(\frac{x}{\sqrt{\nu t}}
  \Bigr)\,,\qquad u(x,t) \,=\, \frac{\Gamma}{\sqrt{\nu t}}\,
  v^G\Bigl(\frac{x}{\sqrt{\nu t}}\Bigr)\,,
\end{equation}
where the vorticity and velocity profiles have the following 
explicit expressions
\begin{equation}\label{Gdef}
  G(\xi) \,=\, \frac{1}{4\pi}\,e^{-|\xi|^2/4}\,, \qquad 
  v^G(\xi) \,=\, \frac{1}{2\pi}\,\frac{\xi^\perp}{|\xi|^2}
  \Bigl(1 - e^{-|\xi|^2/4}\Bigr)\,, \qquad \xi \in \R^2~.
\end{equation}
In particular we have $v^G = K_{BS}*G$. The parameter $\Gamma = \int_{\R^2} 
\omega(x,t)\dd x \in \R$ is called the {\em total circulation} of 
the vortex. We are especially interested in rapidly rotating 
vortices, where the circulation $|\Gamma|$ is much larger than the 
kinematic viscosity $\nu$. This is the regime that is most relevant 
for applications to two-dimensional turbulent flows. 

To study the stability of the vortices \eqref{Osdef}, we look for 
solutions of \eqref{omeq} in the form
\begin{equation}\label{wvdef}
\begin{aligned}
  \omega(x,t) \,&=\, \frac{1}{t+t_0} ~w\left(\frac{x-x_0}{\sqrt{\nu
  (t+t_0)}}\,,\,\log\Bigl(1+\frac{t}{t_0}\Bigr)\right)~, \\[1mm]
  u(x,t) \,&=\, \sqrt{\frac{\nu}{t+t_0}} ~v \left(\frac{x-x_0}{ 
  \sqrt{\nu(t+t_0)}}\,,\,\log\Bigl(1+\frac{t}{t_0}\Bigr)\right)~,
\end{aligned}
\end{equation}
see \cite{GR,Ga2}. The parameters $x_0 \in \R^2$ and $t_0 > 0$ 
are free at this stage, but convenient choices will be made 
later to optimize the results. The new space and time variables 
are denoted by
\begin{equation}\label{xitaudef}
  \xi \,=\, \frac{x-x_0}{\sqrt{\nu(t+t_0)}} \,\in\, \R^2~, 
  \qquad \tau \,=\, \log\Bigl(1+\frac{t}{t_0}\Bigr)\,\ge\, 0~.
\end{equation}
The rescaled vorticity $w(\xi,\tau)$ satisfies the evolution equation
\begin{equation}\label{weq}
  \partial_\tau w + v\cdot\nabla_\xi\,w \,=\, \Delta_\xi\,w + \frac12
  \,\xi\cdot\nabla_\xi\,w + w\,,
\end{equation}
where all dependent and independent variables are now dimensionless. 
The rescaled velocity $v$ is again given by the Biot-Savart law
\eqref{BS}, namely $v(\cdot,\tau) = K_{BS} * w(\cdot,\tau)$.  By
construction, the self-similar solutions \eqref{Osdef} of the original
equation \eqref{omeq} correspond to the family of equilibria
$\{\alpha G \,|\, \alpha \in \R\}$ of the rescaled equation
\eqref{weq}, where $\alpha = \Gamma/\nu$ is the circulation Reynolds
number. Our main purpose is to study the stability of these equilibria
in the large Reynolds number limit $|\alpha| \to \infty$.

Our first task is to choose a suitable function space for the solutions
of \eqref{weq}. There are in principle several possibilities, see 
\cite{GW2,GM2}, but to obtain uniform stability results in the large 
circulation limit it seems necessary to use the Hilbert space 
$X = L^2(\R^2,G^{-1}\dd\xi)$, equipped with the scalar product
\begin{equation}\label{Xscalar}
  \langle w_1,w_2\rangle_X \,=\, \int_{\R^2} G(\xi)^{-1} 
  \,\overline{w_1(\xi)}\,w_2(\xi)\dd\xi\,.
\end{equation}
The associated norm will be denoted by $\|w\|_X$, or simply by 
$\|w\|$ when no confusion is possible. The solutions of \eqref{weq}
are of course real-valued, but the spectral analysis of the 
linearized operator at Oseen's vortex will be performed in the 
complexified space defined by the scalar product \eqref{Xscalar}. 
Since $w \in X$ if and only if $G^{-1/2}w \in L^2(\R^2)$, it 
follows from \eqref{Gdef} that all elements of $X$ have a Gaussian
decay, in the $L^2$ sense, as $|\xi| \to \infty$. In particular 
we have $X \hookrightarrow L^p(\R^2)$ for any $p \in [1,2]$, and 
for later use we introduce the following closed subspaces\:
\begin{align}\label{X0def}
  X_0 \,&=\, \Bigl\{w \in X \,\Big|\, \int_{\R^2} w(\xi)\dd\xi = 
  0\Bigr\}\,,\\ \label{X1def}
  X_1 \,&=\, \Bigl\{w \in X_0 \,\Big|\, \int_{\R^2}\xi_i\,w(\xi)\dd\xi = 0
  \hbox{ for }i=1,2\Bigr\}\,.
\end{align}

We next recall that the Cauchy problem for equation \eqref{weq} is
globally well-posed in the space $X$, and that all solutions converge
to the family of equilibria $\{\alpha G\}_{\alpha \in \R}$ as
$\tau \to +\infty$.

\begin{prop}\label{prop:globex}{\bf\cite{GW2}}
For any $w_0 \in X$, the rescaled vorticity equation \eqref{weq} has 
a unique global mild solution $w \in C^0([0,\infty),X)$ such that 
$w(0) = w_0$. This solution satisfies $\|w(\tau) - \alpha G\|_X \to 0$ 
as $\tau \to +\infty$, where $\alpha = \int_{\R^2} w_0(\xi)\dd \xi$. 
\end{prop}

According to the usual terminology, a mild solution of \eqref{weq} is a 
solution of the associated integral equation, which is Eq.~\eqref{wint} 
below. Proposition~\ref{prop:globex} is essentially taken from  
\cite{GW2}, except that we use here a different function space. 
In \cite{GW2}, the rescaled vorticity equation \eqref{weq} is studied 
in the polynomially weighted space $L^2(m) = \{w\,|\,(1{+}|\xi|^2)^{m/2}w 
\in L^2(\R^2)\}$ with $m > 1$, and it is asserted in \cite{GR,GM1,GM2}, 
without detailed justification, that the results of \cite{GW2} 
remain valid in the smaller space $X$. Since the choice of the
Gaussian space $X$ seems essential in the present paper, we give 
a short proof of Proposition~\ref{prop:globex} in Section~\ref{subsec5.1}
below. 

We now study the behavior of the solutions of \eqref{weq} in the
neighborhood of the family of Oseen vortices. The following
preliminary result shows that the equilibrium $\alpha G$ is
asymptotically stable for any value of the circulation parameter
$\alpha$, and provides a uniform estimate on the size of the basin of
attraction.

\begin{prop}\label{prop:locstab}{\bf\cite{GR,Ga2}} There exists $\epsilon > 0$ 
such that, for all $\alpha \in \R$ and all $w_0 \in \alpha G + X_0$ such 
that $\|w_0 - \alpha G \|_X \le \epsilon$, the unique solution of \eqref{weq} 
with initial data $w_0$ satisfies
\begin{equation}\label{wlocstab}
  \|w(\cdot,\tau) - \alpha G \|_X \,\le\, \min(1,2\,e^{-\tau/2}) 
  \|w_0 - \alpha G \|_X\,, \qquad \forall\tau \ge 0\,.
\end{equation}
\end{prop}

The assumption that the initial perturbation has zero average, namely
that $w_0 - \alpha G \in X_0$, does not restrict the generality\: as
is shown in Section~\ref{subsec4.1}, the general case can be deduced
by an elementary transformation. Estimate \eqref{wlocstab} is
established in \cite[Proposition~4.1]{GR} or in \cite[Proposition~4.5]{Ga2}, 
but the proof is quite simple and for the reader's convenience we 
reproduce it at the end of Section~\ref{subsec5.1}. 

The limitation of Proposition~\ref{prop:locstab} is that it does not
take into account the enhanced dissipation effect for large $|\alpha|$
due to the differential rotation. When translated back into the
original variables, it simply asserts that small perturbations of the
Lamb-Oseen vortex decay to zero on the diffusion time scale.  Building
on the recent work of Li, Wei, and Zhang \cite{LWZ}, we now formulate
an improved stability result which shows that the basin of attraction
of Oseen's vortex becomes very large in the high Reynolds number limit
$|\alpha| \to \infty$, and that perturbations relax to axisymmetry in
a much shorter time. For simplicity, we restrict ourselves to 
solutions of \eqref{weq} that satisfy $w(\cdot,\tau) - \alpha G 
\in X_1$ for all $\tau \ge 0$.  This condition is preserved under the
evolution defined by \eqref{weq}, so our assumption means that we 
consider perturbations that do not alter the total circulation $\alpha$ 
nor the first-order moments of the vorticity distribution.
 
Our main result can be stated as follows. 

\begin{thm}\label{thm:main} There exist positive constants 
$C_1$, $C_2$, and $\kappa$ such that, for all $\alpha \in \R$ and all 
initial data $w_0 \in \alpha G + X_1$ such that 
\begin{equation}\label{wbasin}
  \|w_0 - \alpha G \|_X \,\le\, \frac{C_1\,(1+|\alpha|)^{1/6}}{
  \log(2+|\alpha|)}\,
\end{equation}
the unique solution of \eqref{weq} in $X$ given by 
Proposition~\ref{prop:globex} satisfies, for all $\tau \ge 0$, 
\begin{align}\label{wrdecay}
  \|w(\cdot,\tau) - \alpha G \|_X \,&\le\, C_2\,e^{-\tau} \|w_0 - 
  \alpha G \|_X\,, \\ \label{wperpdecay}
  \|(1-P_r)(w(\cdot,\tau) - \alpha G) \|_X \,&\le\, C_2\|w_0 - 
  \alpha G \|_X \,\exp\Bigl(-\frac{\kappa (1+|\alpha|)^{1/3}\tau}{
  \log(2+|\alpha|)}\Bigr)\,,
\end{align}
where $P_r$ is the orthogonal projection in $X$ onto the subspace 
of all radially symmetric functions. 
\end{thm}

\begin{rems}\label{mainrem}\quad \\[1mm]
{\bf 1.} Theorem~\ref{thm:main} improves Proposition~\ref{prop:locstab} 
only when the circulation parameter $|\alpha|$ is sufficiently large. 
In the proof, we shall therefore assume that $|\alpha| \ge \alpha_0 
\gg 1$, in which case we can replace $1+|\alpha|$ and $2+|\alpha|$
by $|\alpha|$ in estimates \eqref{wbasin}, \eqref{wperpdecay}. Note
that, when $w_0 - \alpha G \in X_1$, the bound \eqref{wlocstab}
holds with the factor $e^{-\tau/2}$ replaced by $e^{-\tau}$ \cite{Ga2}, 
and this is why we also have the overall decay rate $e^{-\tau}$ in 
\eqref{wrdecay}. \\[1mm]
{\bf 2.} As is explained in Section~\ref{subsec4.1}, the assumption
that the initial perturbation $w_0 - \alpha G$ belongs to the 
subspace $X_1$ does not restrict the generality if $\alpha \neq 0$, 
because the general case can be reduced to that situation by a simple change 
of variables. From a more conceptual point of view, if the initial vorticity 
distribution $\omega_0$ has a nonzero total circulation $\Gamma
= \alpha \nu$, and if we choose the parameter $x_0$ in \eqref{wvdef} to 
be the {\em center of vorticity}, then by construction the rescaled 
vorticity satisfies $w - \alpha G \in X_1$ at initial time 
$\tau = 0$, hence for all subsequent times since both the total 
circulation and the center of vorticity are conserved quantities 
for the two-dimensional Navier-Stokes equations. \\[1mm]
{\bf 3.} Estimate \eqref{wbasin} shows that the size of the immediate
basin of attraction of Oseen's vortex $\alpha G$ grows at least like
$|\alpha|^{1/6}$ as $|\alpha| \to \infty$, up to a logarithmic
factor. By ``immediate basin of attraction'', we mean here the set of
initial data for which the solutions converge back to Oseen's vortex
in such a way that they can be controlled by the linearized equation
for all positive times, using for instance Duhamel's formula. We
recall that, according to Proposition~\ref{prop:globex}, all solutions
of \eqref{weq} with initial data $w_0 \in \alpha G + X_0$ converge to
$\alpha G$ as $\tau \to +\infty$, so in this sense the basin of
attraction of Oseen's vortex $\alpha G$ has infinite size for any
fixed $\alpha$.  But Proposition~\ref{prop:globex} is nonconstructive,
and general solutions of \eqref{weq} can go through all the stages of
two-dimensional freely decaying turbulence before reaching the
asymptotic regime described by Oseen's vortex, whereas
Theorem~\ref{thm:main} provides explicit control on the solutions for
all times, as illustrated in estimates
\eqref{wrdecay}, \eqref{wperpdecay}. \\[1mm]
{\bf 4.} The decay rate in \eqref{wperpdecay} is a direct consequence
of the resolvent estimate obtained in \cite{LWZ}, which is known 
to be sharp, hence there are reasons to believe that the bound
\eqref{wperpdecay} is close to optimal. As for the size of the 
basin of attraction, although the proof of Theorem~\ref{thm:main} 
naturally leads to estimate \eqref{wbasin}, we do not know if 
that bound is optimal in any sense. 
\\[1mm]
{\bf 5.} Using the spectral estimate established in \cite[Section~6]{LWZ},
it is possible to show that the solutions of \eqref{weq} considered
in Theorem~\ref{thm:main} satisfy
\begin{equation}\label{wrapid}
  \limsup_{\tau \to +\infty}\frac{1}{\tau}\,\log \|(1-P_r)(w(\cdot,\tau) 
  - \alpha G) \|_X \,\le\, -\kappa' (1+|\alpha|)^{1/2}\,,
\end{equation}
for some constant $\kappa'$ independent of $\alpha$, see also 
inequality \eqref{semiest4} below. However the asymptotic 
regime described in \eqref{wrapid} is only reached at very 
large times, and may not be observable in real flows.
\end{rems}

It is instructive to compare the conclusions of Theorem~\ref{thm:main}
with predictions that can be found in the physical literature 
\cite{BL,BG2}, and with rigorous results describing the 
asymptotic stability of the two-dimensional Couette flow and 
other viscous shear flows \cite{BMV,BVW,BCZ}. To do that, it 
is convenient to fix the circulation and the spatial extent 
of the vortex, while choosing the viscosity parameter small 
enough to reach the high Reynolds number regime. We thus
consider the following initial data for the original vorticity 
equation \eqref{omeq}\:
\[
  \omega_0(x) \,=\, \frac{\Gamma}{R^2}\,G\Bigl(\frac{x}{R}\Bigr)
  + \tilde \omega_0(x)\,, \qquad x \in \R^2\,,
\]
where the circulation $\Gamma > 0$ and the vortex radius $R > 0$ 
are fixed parameters. We assume that the perturbation $\tilde 
\omega_0$ decays rapidly enough at infinity so that $\exp(|x|^2/(8R^2))
\tilde \omega_0 \in L^2(\R^2)$, and satisfies
\[
  \int_{\R^2} \tilde \omega_0(x) \dd x \,=\, 0\,, \qquad
  \int_{\R^2} x_1 \,\tilde \omega_0(x) \dd x \,=\, \int_{\R^2} x_2 
  \,\tilde \omega_0(x) \dd x \,=\, 0\,.
\]
We introduce the turnover time $T = R^2/\Gamma$, and the diffusion
time $t_0 = R^2/\nu$ which depends on the viscosity parameter. As
$t_0/T = \Gamma/\nu$ is the circulation Reynolds number, henceforth
denoted by $\alpha$, we are interested in the regime where
$t_0 \gg T$. If we use the change of variables \eqref{wvdef} with
$x_0 = 0$ and $t_0$ as above, the rescaled vorticity at initial time
takes the form $w_0 = \alpha G + \tilde w_0$, where $\tilde w_0(\xi) = 
t_0 \tilde \omega_0(R\xi)$. Thus $\tilde w_0 \in X_1$ by construction, 
and we can apply Theorem~\ref{thm:main} to control the solution of 
\eqref{omeq} in the small viscosity regime. First, we infer from 
\eqref{wbasin} that
\[
  \frac{1}{\Gamma}\int_{\R^2} |\tilde \omega_0(x)|\dd x \,=\, 
  \frac{\nu}{\Gamma}\int_{\R^2} |\tilde w_0(\xi)|\dd \xi 
  \,\le\, C\,\frac{\nu}{\Gamma}\,\|\tilde w_0\|_X \,\le\, 
  C\,\frac{\nu}{\Gamma}\,|\alpha|^{1/6} \,=\, C \Bigl(\frac{
  \nu}{\Gamma}\Bigr)^{5/6}\,,
\]
and this indicates that we can consider perturbations $\tilde \omega_0$ 
whose size shrinks to zero like $\nu^\gamma$ as $\nu \to 0$, for any 
$\gamma > 5/6$. In the terminology of \cite{BVW}, we have shown that
the {\em stability threshold} for the Oseen vortex in the space $X$
is not larger than $5/6$. Next, we  deduce from \eqref{wperpdecay} that 
the non-axisymmetric part of the perturbation $\tilde w$ disappears 
in a time $\tau$ of the order of $\alpha^{-1/3}$, and since $\tau = 
\log(1 + t/t_0)$ we conclude that the vortex relaxes to axisymmetry 
in a time proportional to
\[
  t_{relax} \,=\, C\,\frac{t_0}{\alpha^{1/3}} \,=\, C\,T\,\frac{t_0}{T}\,
  \Bigl(\frac{\nu}{\Gamma}\Bigr)^{1/3} \,=\, C\,T\,\Bigl(\frac{\Gamma}{
  \nu}\Bigr)^{2/3}\,.
\]
The relaxation time predicted by Theorem~\ref{thm:main} is thus 
proportional to $T\alpha^{2/3}$, where $T$ is the turnover time 
of the unperturbed vortex, and $\alpha = \Gamma/\nu$ is the 
circulation Reynolds number. 

In contrast, the calculations performed in \cite{BL,BG1} indicate that
initial perturbations located near the vortex core, where the
differential rotation is maximal, relax to axisymmetry in a time
proportional to $T\alpha^{1/3}$. Similarly, perturbations of the
two-dimensional Couette flow relax to a shear flow in a time of the
order of $T Re^{1/3}$ \cite{BMV,BVW}, where $Re$ is the Reynolds
number and $T$ is again an appropriate turnover time. The apparent
discrepancy with the conclusions of Theorem~\ref{thm:main} is entirely
due to the fact that the differential rotation of Oseen's vortex
vanishes at the origin and at infinity. Such a degeneracy certainly
does not exist for the shear flows considered in \cite{BMV,BVW}, which
are close to Couette. For the vortex problem, the relaxation time is
known to be proportional to $T\alpha^{1/2}$ for perturbations
initially located at the vortex center \cite{BBG}, and to the
diffusive time $T\alpha$ for perturbations very far away from
the center \cite{RY}. From a mathematical point of view, these various
regimes cannot be considered separately, because they are coupled even
at the level of the linearized equation. If translation invariant
norms are used, this means that for perturbations of the Oseen 
vortex the diffusive decay rate is optimal in general. To obtain a 
stability result that takes advantage of the differential rotation, we use in 
Theorem~\ref{thm:main} the Gaussian space $X$ whose weight
creates an artificial damping of the perturbations far away from the 
origin, see Section~\ref{subsec2.5} below for a more precise 
description of that effect. This is why we can obtain a uniform 
relaxation time $T \alpha^{2/3}$ that is substantially smaller than 
the diffusion time scale, not only for the linearized equation 
but even for the full nonlinear problem. 

To conclude this introduction, we briefly mention that phenomena such
as relaxation of vortices to axisymmetry, or stability of shear flows,
can also be studied for perfect fluids, where the situation is quite
different (and in some sense more complicated). The interested reader 
is referred to \cite{BG1,BG3,BoM} for a physical analysis and 
to \cite{BM,WZZ} for recent mathematical results.

The rest of this paper is organized as follows. In Section~\ref{sec2}
we summarize what is known about the linearized operator at Oseen's
vortex, and we recall the beautiful resolvent estimate recently
obtained by Li, Wei, and Zhang \cite{LWZ}, which allows us to derive
semigroup estimates in Section~\ref{sec3}. The analysis of the
nonlinear problem is postponed to Section~\ref{sec4}, which contains
the proof of Theorem~\ref{thm:main}. The final section is an 
appendix, where we collect some known results on the rescaled 
vorticity equation \eqref{weq}, and derive estimates on the 
Biot-Savart law \eqref{BS} that are needed in our analysis. 

\medskip\noindent{\bf Acknowledgements.} The author is indebted 
to Jacob Bedrossian and David Dritschel for fruitful discussions. 
His research is supported in part by the ANR grant Dyficolti
(ANR-13-BS01-0003-01) from the French Ministry of Research. 

\section{Resolvent estimates for the linearized operator}\label{sec2}

In this section we study the linearization of equation \eqref{weq} at the 
equilibrium point $w = \alpha G$, for a given $\alpha \in \R$. Setting 
$w = \alpha G + \tilde w$ and $v = \alpha v^G + \tilde v$, where 
$\tilde v = K_{BS}*\tilde w$, we obtain for the perturbation $\tilde w$ 
the evolution equation
\begin{equation}\label{tildeweq}
   \partial_\tau \tilde w + \tilde  v \cdot \nabla \tilde w \,=\, 
   (\cL - \alpha \Lambda)\tilde w\,, 
\end{equation}
where $\cL$ is the linear operator in the right-hand side of \eqref{weq}\:
\begin{equation}\label{cLdef}
  \cL w \,=\, \Delta w + \frac12\,\xi\cdot\nabla w + w\,,
\end{equation}
and $\Lambda$ is the linearization at $w=G$ of the quadratic term $v \cdot 
\nabla w = (K_{BS}*w) \cdot \nabla w$\:
\begin{equation}\label{Lambdadef}
  \Lambda w \,=\, v^G \cdot \nabla w + (K_{BS}*w) \cdot \nabla G \,\equiv\,
  \Lamad w + \Lamnl w\,.
\end{equation}
Here and in what follows, it is understood that all differential operators 
act on the space variable $\xi \in \R^2$, except for the time derivative 
$\partial_\tau$ which is always explicitly indicated. 

We first recall a few classical properties of the operators $\cL$ and
$\Lambda$, which can be found e.g. in \cite{GW1,GW2,Ma1,Ga2,GM2}. 
We only consider the situation where these operators act on the Hilbert 
space $X = L^2(\R^2,G^{-1}\dd\xi)$, equipped  with the scalar product 
\eqref{Xscalar}, but similar results in larger function spaces can be found 
in \cite{GW2,GM2}. Our goal is to present the optimal resolvent 
bounds obtained by Li, Wei, and Zhang \cite{LWZ} for the linearized operator 
$\cL - \alpha\Lambda$ in the fast rotation limit $|\alpha| \to +\infty$.
These estimates will serve as a basis for all developments in 
Sections~\ref{sec3} and~\ref{sec4}. 

\subsection{Fundamental properties of $\cL$ and $\Lambda$}
\label{subsec2.1}
The first observation is that the operator $\cL$ is {\em selfadjoint} 
in the space $X$, with compact resolvent and purely discrete 
spectrum
\begin{equation}\label{cLspec}
  \sigma(\cL) \,=\, \Bigl\{-\frac{n}{2}\,\Big|\, n = 0,1,2,\dots\Bigr\}\,.
\end{equation}
Indeed, a formal calculation shows that $\cL$ is conjugated to the Hamiltonian 
of the harmonic oscillator in $\R^2$\: 
\begin{equation}\label{Lconj}
  L \,:=\, G^{-1/2} \,\cL ~G^{1/2} \,=\, \Delta \,-\, \frac{|\xi|^2}{16}
  \,+\, \frac12\,.
\end{equation}
As is well known (see e.g. \cite{He}), the operator $L$, when defined on its 
maximal domain, is selfadjoint in $L^2(\R^2)$ with compact resolvent and 
spectrum given by \eqref{cLspec}. This proves the desired properties of the 
operator $\cL$ in $X = L^2(\R^2,G^{-1}\dd\xi)$, and we also obtain in 
this way the following characterization of its domain\:
\[
  D(\cL) \,=\, \bigl\{w \in X \,|\, \cL w \in X\bigr\} \,=\, 
  \bigl\{w \in X \,|\, \Delta w \in X,~ (1{+}|\xi|)\nabla w \in X,~ 
   (1{+}|\xi|)^2 w \in X\bigr\}\,.
\]
Concerning the eigenproperties of $\cL$, we mention that the kernel
$\ker(\cL)$ is the one-dimensional subspace spanned by the Gaussian
vorticity profile $G$, and is orthogonal in $X$ to the hyperplane
$X_0$ defined by \eqref{X0def}. The second eigenvalue $-1/2$ has 
multiplicity two, with eigenfunctions given by the first order derivatives 
$\partial_i G = -\frac12 \xi_i G$ for $i = 1,2$, and for later use 
we observe that the orthogonal complement of the spectral subspace 
spanned by $G,\partial_1 G,\partial_2 G$ is precisely the 
subspace $X_1$ defined by \eqref{X1def}. More generally, for any
$k \in \N$, the eigenvalue $-k/2$ has multiplicity $k+1$ and the
corresponding eigenspace is spanned by Hermite functions of order $k$,
namely homogeneous $k^{\rm th}$ order derivatives of the Gaussian
profile $G$ \cite{GW1}. 

The second key observation is that the operator $\Lambda$ is a 
{\em relatively compact} perturbation of $\cL$, which is 
{\em skew-symmetric} in $X$\:
\begin{equation}\label{Lamskew}
  \langle \Lambda w_1\,,\,w_2\rangle_X + \langle w_1\,,\,\Lambda 
  w_2\rangle_X \,=\, 0~, \qquad \hbox{for all } w_1,w_2 \in 
  D(\cL)\,.
\end{equation}
Indeed, if $\Lambda$ is decomposed as in \eqref{Lambdadef}, 
the advection term $\Lamad = v^G \cdot \nabla$ is a first order 
operator with smooth coefficients that decay to zero at infinity, hence is a
relatively compact perturbation of the second order elliptic operator
$\cL$. Similarly, it is straightforward to verify that the nonlocal operator
$\Lamnl$ in \eqref{Lambdadef} is compact in $X$, hence also 
relatively compact with respect to $\cL$. On the other hand, since the 
velocity field $v^G$ is divergence-free and satisfies $\xi \cdot v^G(\xi) 
= 0$ for all $\xi \in \R^2$, it is clear that $\div(G^{-1}v^G) = 0$, and this 
implies that the operator $\Lamad$ is skew-symmetric in $X$. The 
corresponding property for $\Lamnl$ can be established by a direct 
calculation, which takes into account the structure of the Biot-Savart kernel, 
see \cite{GW2,Ga2}. In fact, it is shown in \cite{Ma1} that the operator 
$\Lambda$ is even {\em skew-adjoint}  in $X$ if it is defined on its 
maximal domain $D(\Lambda) = \{w \in X \,|\, \Lambda w \in X\}$. 

Another useful result of \cite{Ma1} is the following characterization 
of the kernel of $\Lambda$ in $X$\:
\begin{equation}\label{KerLambda}
  \ker(\Lambda) \,=\, Y_0 \oplus \bigl\{\beta_1 \partial_1 G + \beta_2 
  \partial_2 G\,|\, \beta_1,\beta_2 \in \R\bigr\}\,,
\end{equation}
where $Y_0 \subset X$ denotes the subspace of all radially symmetric
functions. Indeed, it is clear by symmetry that $\Lambda$ vanishes on
any radially symmetric function in $X$, so that $Y_0 \subset \ker(\Lambda)$. 
Moreover, if we differentiate the identity $v^G\cdot\nabla G = (K_{BS} * G)
\cdot \nabla G = 0$ with respect to $\xi_1$ and $\xi_2$, we see that 
$\Lambda(\partial_i G) = 0$ for $i = 1,2$. Thus $\ker(\Lambda)$ contains 
the right-hand side of \eqref{KerLambda}, and the converse inclusion is 
established in \cite{Ma1} using the decomposition presented in 
Section~\ref{subsec2.4} below and the explicit form of the one-dimensional 
operator $\Lambda_n$ in \eqref{Lambdandef}, see also \cite{Ga2}. 

\subsection{Dissipativity and linear stability}
\label{subsec2.2}
We recall that an operator $A : D(A) \to X$ is {\sl dissipative} if 
$\Re\,\langle Aw,w\rangle_X \le 0$ for all $w \in D(A)$, or equivalently 
if $\|(\lambda - A)w\| \ge \lambda \|w\|$ for all $w \in D(A)$ and 
all $\lambda > 0$ \cite{Pa}. The operator $A$ is called $m$-dissipative
if in addition any $\lambda > 0$ belongs to the resolvent set of $A$ 
\cite{Ka}. By the Lumer-Phillips theorem, an operator $A$ is $m$-dissipative
if and only if it generates a strongly continuous semigroup of contractions 
in $X$ \cite{Pa}. 

The properties collected in Section~\ref{subsec2.1} readily imply the 
following important result\:

\begin{prop}\label{prop:dissip} {\bf \cite{GW2,Ga2}} For any 
$\alpha \in \R$\: \\[1mm]
a) the operator $\cL - \alpha \Lambda$ is $m$-dissipative in $X$; \\[1mm]
b) the operator $\cL - \alpha \Lambda + \frac12$ is $m$-dissipative 
   in $X_0$; \\[1mm]
c) the operator $\cL - \alpha \Lambda + 1$ is $m$-dissipative in $X_1$.
\end{prop}
 
Proposition~\ref{prop:dissip} shows in particular that the Oseen
vortex $\alpha G$ is a {\em linearly stable} equilibrium of the
rescaled vorticity equation \eqref{weq}, for any value of the
circulation Reynolds number $\alpha \in \R$. In addition, if we
restrict ourselves (without loss of generality) to perturbations in
the invariant subspace $X_0$, the linearized operator $\cL - \alpha 
\Lambda$ has a {\em uniform spectral gap} (of size $1/2$) for all 
$\alpha \in \R$. As the nonlinearity in \eqref{tildeweq} does not 
involve the parameter $\alpha$, this implies a uniform lower bound
on the size of the (immediate) basin of attraction of the vortex, 
as asserted in Proposition~\ref{prop:locstab}. In the invariant 
subspace $X_1 \subset X_0$, the spectral gap is even larger (of size $1$), 
and the perturbations therefore decay to zero like $e^{-\tau}$ as 
$\tau \to +\infty$. More details can be found in Section~\ref{subsec5.1}, 
which contains in particular a short proof of Proposition~\ref{prop:locstab}. 

\subsection{Enhanced dissipation for large circulation $\alpha$}
\label{subsec2.3}
The main purpose of the present paper is to investigate what can be
said beyond Proposition~\ref{prop:dissip}, using the enhanced
dissipation properties of the linearized operator $\cL - \alpha\Lambda$
for large values of $\alpha$. To do that, it is obviously necessary to
restrict ourselves to perturbations in the orthogonal complement
$\ker(\Lambda)^\perp$, because on $\ker(\Lambda)$ the linearized
operator $\cL - \alpha \Lambda$ reduces to $\cL$ and does not depend
on $\alpha$. Using \eqref{KerLambda} is it easy to verify that the
subspace $\ker(\Lambda)^\perp$ is invariant under the actions of both
$\cL$ and $\Lambda$, and that $\ker(\Lambda)^\perp \subset X_1$ where
$X_1$ is defined in \eqref{X1def}. Proposition~\ref{prop:dissip} thus 
shows that  $\cL - \alpha \Lambda + 1$ is $m$-dissipative in 
$\ker(\Lambda)^\perp$ for all $\alpha \in \R$, but much more is known 
for large values of $|\alpha|$. The following resolvent estimate 
is the main result of the paper by Li, Wei, and Zhang\:

\begin{prop}\label{prop:LWZ} {\bf \cite{LWZ}} There exist 
positive constants $c_1, c_2$ such that, for all $\alpha \in \R$,  
\begin{equation}\label{LWZest}
  c_1 (1+|\alpha|)^{-1/3} \,\le\, \sup_{\lambda \in \R} 
  \|(\cL - \alpha \Lambda - i\lambda)^{-1}\|_{X_\perp \to 
  X_\perp} \,\le\, c_2 (1+|\alpha|)^{-1/3}\,,
\end{equation}
where $X_\perp = \ker(\Lambda)^\perp \subset X$.  
\end{prop}

\begin{rem}\label{normrem}
Here and what follows, if $Y$ is a Banach space, we denote by 
$\|B\|_{Y \to Y}$ the operator norm of any bounded linear map 
$B : Y \to Y$. 
\end{rem}

Since $\Lambda$ is a relatively compact perturbation of the operator
$\cL$, which itself has compact resolvent, it is clear that the
linearized operator $\cL - \alpha\Lambda$ has compact resolvent in $X$
for any $\alpha \in \R$. In particular, the spectrum
$\sigma(\cL - \alpha\Lambda)$ is a sequence of complex eigenvalues
$\lambda_k(\alpha)$, where $k \in \N$, and it is not difficult to
verify that $\Re(\lambda_k(\alpha)) \to -\infty$ as $k \to \infty$.
Moreover, Proposition~\ref{prop:dissip} shows that
$\Re \lambda_k(\alpha) \le 0$ for all $k \in \N$, and that
$\Re \lambda_k(\alpha) \le -1$ if we only consider eigenvalues
corresponding to eigenfunctions in the invariant subspace $X_1$.
Again, much more is known if we restrict ourselves to the smaller
subspace $\ker(\Lambda)^\perp \subset X_1$. To formulate that, we
define for any $\alpha \in \R$ the spectral lower bound
\begin{equation}\label{Sigmadef}
  \Sigma(\alpha) \,=\, \inf\bigl\{\Re(z)\,\big|\, z \in \sigma(
  -\cL_\perp + \alpha \Lambda_\perp) \bigr\}\,,
\end{equation}
where $\cL_\perp, \Lambda_\perp$ denote the restrictions of $\cL,\Lambda$ 
to $X_\perp = \ker(\Lambda)^\perp$. Then Proposition~\ref{prop:LWZ} implies 
that $\Sigma(\alpha) \ge c_2^{-1} (1+|\alpha|)^{1/3}$ for all $\alpha 
\in \R$, because for any linear operator $A$ in $X$ one has the inequality
\begin{equation}\label{resnorm}
  \|(A - z)^{-1}\| \,\ge\, \frac{1}{\dist(z,\sigma(A))}\,,
  \qquad \hbox{for all } z \in \C \setminus \sigma(A)\,.
\end{equation}
In fact, another result of Li, Wei, and Zhang provides an improved 
lower bound on $\Sigma(\alpha)$\:

\begin{prop}\label{prop:LWZ2} {\bf \cite{LWZ}} There exists a 
positive constant $c_3$ such that $\Sigma(\alpha) \ge c_3 
(1+|\alpha|)^{1/2}$ for all $\alpha \in \R$. 
\end{prop}

According to Proposition~\ref{prop:LWZ2}, the eigenvalues 
$\lambda_k(\alpha)$ of the linearized operator $\cL - \alpha\Lambda$ 
are either independent of $\alpha$, because the corresponding 
eigenfunctions lie in the kernel of $\Lambda$, or have real parts
that converge to $-\infty$ at least as fast as $-|\alpha|^{1/2}$ 
when $|\alpha| \to \infty$. This is in full agreement with 
the numerical calculations of Prochazka and  Pullin \cite{PP1,PP2}, 
which indicate that the rate $\cO(|\alpha|^{1/2})$ is indeed
optimal. Note also that the spectral lower bound $\Sigma(\alpha)$ 
is much larger, when $|\alpha| \gg 1$, than what can be predicted 
from the pseudospectral estimate \eqref{LWZest}, and this is due
to the fact that the linearized operator $\cL - \alpha\Lambda$ 
is highly non-selfadjoint in that regime. Indeed, it is well-known 
that equality holds in \eqref{resnorm} if $A$ is a selfadjoint
(or normal) operator in $X$ \cite{Ka}. 

\subsection{Fourier decomposition and reduction to 
one-dimensional operators}\label{subsec2.4}
For later use, we briefly describe one important step in the proof of 
Proposition~\ref{prop:LWZ}. Our starting point is the observation that 
both operators $\cL$ and $\Lambda$ are invariant under rotations
about the origin in $\R^2$. To fully exploit this symmetry, it is
useful to introduce polar coordinates $(r,\theta)$ in the plane and to
expand the vorticity and the velocity field in Fourier series with
respect to the angular variable $\theta \in \mathbb{S}^1$. In this
way, our function space $X$ is decomposed into a direct sum\:
\begin{equation}\label{Xdecomp}
   X \,=\, \mathop{\oplus}\limits_{n \in \Z} Y_n\,,
\end{equation}
where $Y_n = \{w \in X \,|\, e^{-in\theta} w\hbox{ is radially symmetric}\}$. 
The crucial point is that, for each $n \in \Z$, the closed subspace 
$Y_n$ is invariant under the action of both linear operators $\cL$ and 
$\Lambda$. As is shown in \cite{GW2}, the restriction $\cL_n$ of $\cL$ 
to $Y_n$ is the one-dimensional operator
\begin{equation}\label{cLndef}
   \cL_n \,=\, \partial_r^2 + \Bigl(\frac{r}{2}+\frac{1}{r}
   \Bigr)\partial_r + \Bigl(1 - \frac{n^2}{r^2}\Bigr)~, 
\end{equation}
which is defined on the positive half-line $\{r > 0\}$, with
homogeneous Dirichlet condition at the origin if $n = 0$ or $|n| \ge
2$, and homogeneous Neumann condition if $|n| \ge 1$.  Similarly,
the restriction $\Lambda_n$ of $\Lambda$ to $Y_n$ vanishes for
$n = 0$ and is given by
\begin{equation}\label{Lambdandef} 
  \Lambda_n w \,=\, in\bigl(\phi w - g \Omega_n[w]\bigr)\,, 
  \qquad\hbox{for } n \neq 0\,,
\end{equation}
where $\phi, g$ are the functions on $\R_+$ defined by
\begin{equation}\label{phigdef}
   \phi(r) \,=\, \frac{1}{2\pi r^2}(1-e^{-r^2/4})~, \qquad
   g(r) \,=\, \frac{1}{4\pi}\,e^{-r^2/4}~, \qquad r > 0~,
\end{equation}
and $\Omega_n = \Omega_n[w]$ is the unique solution of the differential 
equation $-\Omega_n'' -\frac1r \Omega_n' + \frac{n^2}{r^2}\Omega_n = w$ 
on $\R_+$ that is regular at the origin and at infinity, namely
\begin{equation}\label{Omndef}
   \Omega_n(r) \,=\, \frac{1}{4|n|}\left(\int_0^r \Bigl(\frac{r'}{r}
   \Bigr)^{|n|} r'w(r')\dd r' + \int_r^\infty \Bigl(\frac{r}{r'}
   \Bigr)^{|n|} r'w(r')\dd r'\right)~, \quad r > 0~.
\end{equation}

Thanks to the decomposition \eqref{Xdecomp}, to prove Proposition~\ref{prop:LWZ} 
it is sufficient to study the family of one-dimensional operators
\begin{equation}\label{Hndef}
  H_{n,\beta} \,=\, -\cL_n + \alpha \Lambda_n \,\equiv\, -\cL_n + i\beta M_n\,,
  \qquad n \neq 0\,, 
\end{equation}
where $\beta = n \alpha \in \R$ and $M_n w = \phi w - g \Omega_n[w]$. 
When $|n| \ge 2$, these operators act on the Hilbert space $Z = L^2(\R_+,
g^{-1}r\dd r)$, which is the analog of the original space $X$ in polar 
coordinates. When $n = \pm 1$, the operator $M_n$ has a one-dimensional 
kernel spanned by the function $rg$, because $\Omega_n[rg] = r\phi$ if 
$|n| = 1$. In that case, to obtain enhanced dissipation estimates, it is 
necessary to consider $H_{n,\beta}$ as acting on the orthogonal complement 
of the kernel, namely on the hyperplane
\begin{equation}\label{Z0def}
  Z_0 \,=\, \Bigl\{w \in Z \,\Big|\, \int_0^\infty r^2 w(r)\dd r = 0\Bigr\}\,.
\end{equation}
As in Section~\ref{subsec2.1} above, one can verify that the operator $\cL_n$ 
is selfadjoint in $Z$, with $\cL_n \ge |n|/2$ for any $n \in \Z$. Moreover, if 
$n = \pm 1$, then $\cL_n \ge 1$ on $Z_0$. Finally, the bounded operator
$M_n$ is symmetric in $Z$ for any $n \neq 0$. 

Proposition~\ref{prop:LWZ} is a direct consequence of the following
optimal resolvent estimate for the operator $H_{n,\beta}$.

\begin{prop}\label{prop:LWZ3} {\bf \cite{LWZ}} There exist 
positive constants $c_1, c_2$ such that, for any $\beta \in \R$ 
and any $n \in \Z$ with $|n| \ge 2$, the following estimate holds\:
\begin{equation}\label{LWZest2}
  c_1 (1+|\beta|)^{-1/3} \,\le\, \sup_{\lambda \in \R} 
  \|(H_{n,\beta} - i\lambda)^{-1}\|_{Z \to Z} \,\le\, 
  c_2 (1+|\beta|)^{-1/3}\,.
\end{equation}
Moreover, if $n = \pm 1$, we have the same estimate in the subspace $Z_0$\:
\begin{equation}\label{LWZest3}
  c_1 (1+|\beta|)^{-1/3} \,\le\, \sup_{\lambda \in \R} 
  \|(H_{n,\beta} - i\lambda)^{-1}\|_{Z_0 \to Z_0} \,\le\, 
  c_2 (1+|\beta|)^{-1/3}\,.
\end{equation}
\end{prop}

\subsection{Historical and heuristical remarks}
\label{subsec2.5}
Proposition~\ref{prop:LWZ} is the culmination of a series of works
where resolvent estimates similar to \eqref{LWZest} were obtained 
for simplified models. It was first realized that, in the stability
analysis of the Lamb-Oseen vortex, the enhanced dissipation effect
for large values of $|\alpha|$ is due to the interplay of the diffusion 
operator $\cL$ and the advection term $\Lamad = v^G\cdot\nabla$, 
whereas the nonlocal correction $\Lamnl$ plays a relatively minor 
role. As in \eqref{Lconj}, we observe that
\[
   G^{-1/2}\,(\cL-\alpha\Lamad)\,G^{1/2}  \,=\, L - \alpha \Lamad\,,
\]
because the operator $\Lamad$ commutes with the radially symmetric weight 
$G^{1/2}$. In this way, we are thus led to study a large, skew-symmetric 
perturbation of the harmonic oscillator $L$ in $L^2(\R^2)$. In \cite{GGN}, 
I.~Gallagher, F.~Nier and the author analyzed the following complex 
Schr\"odinger operator in $L^2(\R)$\: 
\[
  H_\alpha \,=\, -\partial_x^2 + x^2 + i \alpha \phi(x)\,, \qquad
  x \in \R\,,
\]
which can be considered as a one-dimensional analog of 
$L-\alpha\Lamad$. If $\phi$ is given by \eqref{phigdef}, they 
proved that 
\begin{equation}\label{GGNest}
  \sup_{\lambda \in \R}\|(H_\alpha - i\lambda)^{-1}\| \,=\, 
  \cO(|\alpha|^{-1/3})\,, \quad \hbox{and}\quad \inf
  \bigl\{\Re(z)\,\big|\, z \in \sigma(H_\alpha)\bigr\}
  \,\ge\, \cO(|\alpha|^{1/2})\,,
\end{equation}
as $|\alpha| \to \infty$, and they explained the origin of the exponents 
$1/3$ and $1/2$ appearing in \eqref{GGNest}. Using similar techniques, 
Wen Deng \cite{De1} obtained for the simplified operator 
$\cL - \alpha \Lamad$ the estimate
\[
  c_1 (1+|\alpha|)^{-1/3} \,\le\, \sup_{\lambda \in \R} 
  \|(\cL - \alpha \Lamad - i\lambda)^{-1}\|_{Y_0^\perp \to 
  Y_0^\perp} \,\le\, c_2 (1+|\alpha|)^{-1/3}\,,
\]
where $Y_0^\perp$ is the orthogonal complement in $X$ of all radially
symmetric functions. Subsequently, she also proved that the bound
\eqref{LWZest2} holds for the full operator $H_{n,\beta}$ provided the
azimuthal wavenumber $|n|$ is sufficiently large \cite{De2}. That
restriction was completely removed by Li, Wei, and Zhang in
\cite{LWZ}, using careful estimates which show (roughly speaking) that
the nonlocal term $\Lamnl$ in the skew-symmetric operator $\Lambda$
can be considered as a perturbation of the local differential operator
$\Lamad$. This argument fails when $n = \pm 1$, as is attested by the
existence of a nontrivial element in the kernel of $\Lambda_n$, but in 
that particular case the authors of \cite{LWZ} were able to eliminate 
completely the nonlocal term $\Lamnl$ using a beautiful transformation 
inspired from scattering theory.

It is important to emphasize here the crucial role played by the
Gaussian weight $G^{-1}(\xi)$ in all resolvent estimates presented in
this section.  The analysis of simplified one-dimensional models in
\cite{GGN,De1} shows that, for large values of the circulation
parameter, the resolvent on the imaginary axis is bounded by
$C |\alpha|^{-2/3}$ when acting on perturbations located in the vortex
core, where the differential rotation is maximal, and by
$C |\alpha|^{-1/2}$ for perturbations located at the origin, where the
differential rotation degenerates. When translated back into the
original variables, these partial results indicate that
axisymmetrization occurs in a time proportional to $|\alpha|^{1/3}$
and $|\alpha|^{1/2}$, respectively, in full agreement with the
predictions made in \cite{BL,BG2,BBG}. For arbitrary perturbations,
however, it is clear that no dependence on $|\alpha|$ can be obtained
if one estimates the resolvent in a translation invariant norm,
because the differential rotation of the vortex vanishes at
infinity. The situation is different in the weighted space $X$, where
the diffusion operator is replaced by the harmonic oscillator, see
\eqref{Lconj}. In that case, due to the quadratic potential, the
resolvent is small also for perturbations located far away from the
origin. Summarizing, there is a critical distance to the origin, of
the order of $|\alpha|^{1/6}$ if $|\alpha| \gg 1$, where the enhanced
dissipation due to the differential rotation is of the same order as
the artificial damping due to the quadratic potential in the harmonic
oscillator, and this determines the size of the resolvent in the
weighted space $X$, which is $\cO(|\alpha|^{-1/3})$ as stated in 
Proposition~\ref{prop:LWZ}. We insist on saying that this new exponent
$-1/3$, which predicts axisymmetrization in a time proportional to
$|\alpha|^{2/3}$, may not be directly related to physical phenomena\:
it is rather a consequence of our choice of measuring perturbations in
the Gaussian weighted space $X$.

\section{Semigroup estimates}\label{sec3}

Applying the resolvent bounds established in Propositions~\ref{prop:LWZ} and 
\ref{prop:LWZ3}, we now obtain sharp decay estimates for the semigroup 
generated by the linearized operator $\cL - \alpha\Lambda$ in the subspace 
$X_\perp = \ker(\Lambda)^\perp$. We use the Fourier decomposition introduced
in Section~\ref{subsec2.4} and first consider the semigroup defined 
by the one-dimensional operator \eqref{Hndef} in the space $Z$. 

\begin{prop}\label{prop:semi1} For any $n \in \Z$, $n \neq 0$, and 
any $\beta \in \R$, the linear operator $-H_{n,\beta}$ defined in 
\eqref{Hndef}  is the generator of an analytic semigroup in 
the space $Z = L^2(\R_+,g^{-1}r\dd r)$. Moreover, there exist positive 
constants $c_4, c_5$ such that, for any $n \in \Z$ with $|n| \ge 2$ 
and any $\beta \in \R$ with $|\beta| \ge 1$, the following estimate 
holds\:
\begin{equation}\label{semiest1}
  \|e^{-\tau H_{n,\beta}}\|_{Z \to Z} \,\le\, \min\Bigl(e^{-|n|\tau/2}\,,\,
   c_4 |\beta|^{2/3}\,e^{-c_5 |\beta|^{1/3}\tau}\Bigr)\,, \qquad \tau \ge 0\,.
\end{equation}
If $n = \pm1$, we have a similar estimate in the subspace $Z_0$ defined 
in \eqref{Z0def}\:
\begin{equation}\label{semiest2}
  \|e^{-\tau H_{n,\beta}}\|_{Z_0 \to Z_0} \,\le\, \min\Bigl(e^{-\tau}\,,\,
   c_4 |\beta|^{2/3}\,e^{-c_5 |\beta|^{1/3}\tau}\Bigr)\,, \qquad \tau \ge 0\,.
\end{equation}
\end{prop}

\begin{proof}
We recall that $-H_{n,\beta} = \cL_n - i\beta M_n$, where $\cL_n$ is 
a selfadjoint operator in $Z$ satisfying $\cL_n \le -|n|/2$, and 
$M_n$ is a bounded symmetric operator. By classical perturbation 
theory \cite[Section~3.2]{Pa}, it follows that $-H_{n,\beta}$ 
generates an analytic semigroup in $Z$. Moreover, since 
the operator $-H_{n,\beta} + |n|/2$ is $m$-dissipative, the 
Lumer-Phillips theorem \cite[Section~1.4]{Pa} implies that 
$\|e^{-\tau H_{n,\beta}}\|_{Z \to Z} \le e^{-|n|\tau /2}$ for all $t \ge 0$. 
For later use, we observe that there exists a constant $c_6 > 0$ 
such that $\|M_n\|_{Z \to Z} \le c_6$ for all nonzero $n \in \Z$.
Indeed $M_n w = \phi w - g \Omega_n[w]$ where $|\phi| \le (8\pi)^{-1}$, 
and it follows from \eqref{Omndef} that $|\Omega_n[w]| \le \frac14 
\int_0^\infty r|w(r)|\dd r \le C \|w\|_Z$ for all $n \in \Z$, $n \neq 0$. 
This shows that the {\em numerical range} 
\[
  \cN(H_{n,\beta}) \,=\, \bigl\{\langle w\,,\,H_{n,\beta}w\rangle_Z
   \,\big|\, w \in D(\cL_n)\,,~\|w\|_Z = 1\bigr\}
\]
satisfies
\begin{equation}\label{numrange}
  \cN(H_{n,\beta}) \,\subset\, \bigl\{z \in \C \,\big|\, 
  \Re(z) \ge |n|/2\,,~|\Im(z)| \le c_6|\beta| \bigr\}\,.
\end{equation}

To estimate the semigroup $e^{-\tau H_{n,\beta}}$ for $\tau > 0$ and 
$|\beta| \ge 1$, we use the inverse Laplace formula
\begin{equation}\label{Laplace}
  e^{-\tau H_{n,\beta}} \,=\, \frac{1}{2\pi i}\int_\Gamma (H_{n,\beta}-z)^{-1} 
  \,e^{-z\tau}\dd z\,, 
\end{equation}
where $\Gamma$ is the integration path in the complex plane depicted in 
Fig.~1. More precisely, we define $\Gamma = \Gamma_1 \cup \Gamma_2 \cup 
\Gamma_3$ where
\begin{align}\nonumber
  \Gamma_1 \,&=\,  \Bigl\{x_0 - iy_0 - (1-i)s
  \,\Big|\, -\infty \le s \le 0\Bigr\}\,, \\ \label{Gammadef}
  \Gamma_2 \,&=\, \Bigl\{x_0 + iy\,\Big|\, 
  -y_0 \le y \le y_0\Bigr\}\,, \\ \nonumber
  \Gamma_3 \,&=\, \Bigl\{x_0 + iy_0 + (1+i)s
  \,\Big|\, 0 \le s \le \infty\Bigr\}\,. 
\end{align}
Here $x_0 = |\beta|^{1/3}/(2c_2)$ and $y_0 = 2c_6|\beta|$, where the constants
$c_2, c_6$ are as in \eqref{LWZest2}, \eqref{numrange}. 

\figurewithtex 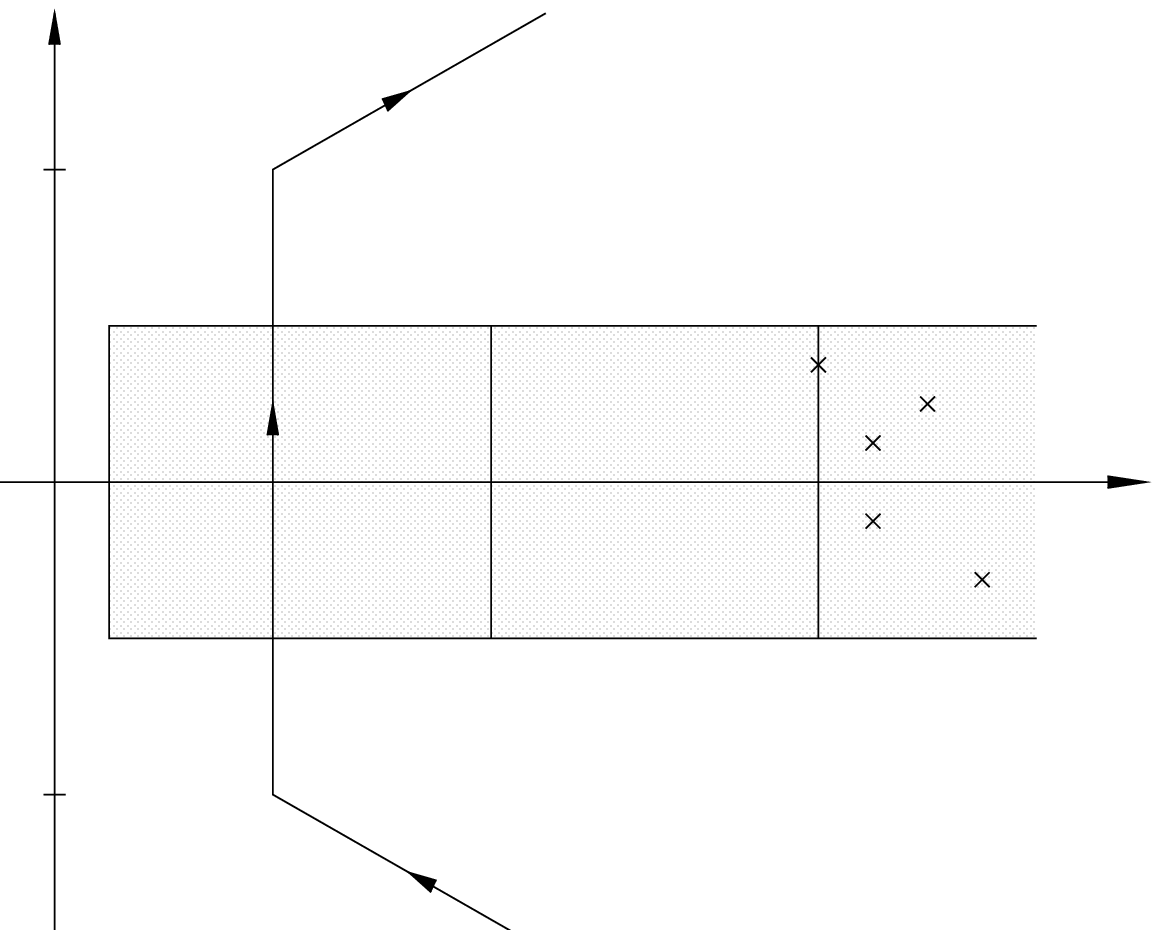 fig1.tex 9.50 11.00
{\bf Fig.~1:} The integration path \eqref{Gammadef} surrounds the 
spectrum of $H_{n,\beta}$, which is discrete and entirely contained 
in the shaded region defined by \eqref{numrange}. The pseudospectral 
abscissa $x_1 = 2x_0 = |\beta|^{1/3}/c_2$ is given by 
Proposition~\ref{prop:LWZ3}, and the spectral lower bound $x_2 = c_3 
|\beta|^{1/2}$ by \cite[Theorem~6.1]{LWZ}, see Remark~\ref{rem:abscissa} 
below. \cr

We first assume that $|n| \ge 2$ and compute the norm of the semigroup
in the whole space $Z$ using formula \eqref{Laplace}\:
\[
 \|e^{-\tau H_{n,\beta}}\|_{Z \to Z} \,\le\, \frac{1}{2\pi}\sum_{j=1}^3 \int_{\Gamma_j} 
 \|(H_{n,\beta}-z)^{-1} \|_{Z \to Z} \,\,e^{-\Re(z)\tau}\,|{\rm d} z| \,=\, I_1(\tau) 
  + I_2(\tau) + I_3(\tau)\,. 
\]
If $z = x_0 + iy \in \Gamma_2$, the standard factorization $H_{n,\beta}-z 
= (H_{n,\beta}-iy)(1 - x_0 (H_{n,\beta}-iy)^{-1})$ combined with estimate 
\eqref{LWZest2} yields the bound
\begin{equation}\label{x0bound}
  \|(H_{n,\beta}-z)^{-1} \| \,\le\, \frac{\|(H_{n,\beta}-iy)^{-1}\|}{1 - 
  |x_0|\,\|(H_{n,\beta}-iy)^{-1}\|} \,\le\, 2 c_2 |\beta|^{-1/3}\,,
\end{equation}
because $|x_0|\,\|(H_{n,\beta}-iy)^{-1}\| \le 1/2$ for any $y \in \R$. It 
follows that 
\[
  I_2(\tau) \,=\, \frac{1}{2\pi} \int_{-y_0}^{y_0}\|(H_{n,\beta}-z)^{-1}\|
  \,e^{-x_0 \tau}\dd y \,\le\, \frac{2y_0c_2}{\pi}\,|\beta|^{-1/3} 
  \,e^{-x_0 \tau} \,=\, \frac{4c_2c_6}{\pi}\,|\beta|^{2/3} \,e^{-\tau|\beta|^{1/3}/(2c_2)}\,.
\]

On the other hand, if $z = x_0 + iy_0 + (1+i)s \in \Gamma_3$, then
$z \notin \cN(H_{n,\beta})$ because $\Im(z) = 2c_6|\beta| + s > c_6|\beta|$. 
It follows that
\[
   \|(H_{n,\beta}-z)^{-1} \| \,\le\, \frac{1}{\dist(z,\cN(H_{n,\beta}))}
   \,\le\, \frac{1}{c_6|\beta| + s}\,,
\]
and we deduce that
\[
  I_3(\tau) \,=\, \frac{1}{2\pi} \int_0^\infty \|(H_{n,\beta}-z)^{-1}\|
  \,e^{-(x_0 +s)\tau}\sqrt{2}\dd s \,\le\, \frac{1}{\sqrt{2}\pi}\,e^{-x_0 \tau}
  \Psi(c_6|\beta|\tau)\,,
\]
where $\Psi : (0,\infty) \to (0,\infty)$ is the decreasing function
defined by
\[
  \Psi(\rho) \,=\, \int_0^\infty \frac{e^{-s}}{\rho + s}\dd s
  \,\sim\, \begin{cases}\log(\rho^{-1}) & \hbox{as } \rho \to 0\,, 
  \\ \rho^{-1} & \hbox{as } \rho \to \infty \,.\end{cases}
\]
If we assume that $|\beta|\tau \ge 1$, we thus obtain 
\[
  I_3(\tau) \,\le\, \frac{\Psi(c)}{\sqrt{2}\pi}\,e^{-x_0 \tau} \,=\, 
  \frac{\Psi(c)}{\sqrt{2}\pi}\,\,e^{-\tau|\beta|^{1/3}/(2c_2)}\,.
\]
The same bound also holds for the integral $I_1(\tau)$, because
$\Gamma_1 = \overline{\Gamma}_3$ and all estimates we used for the
resolvent $(H_{n,\beta}-z)^{-1}$ are unchanged if $z$ is replaced 
by the complex conjugate $\bar z$. 

Summarizing we have shown that, if $|\beta| \ge 1$ and $|\beta| \tau \ge 1$, 
then
\begin{equation}\label{prelim}
  \|e^{-\tau H_{n,\beta}}\|_{Z \to Z} \,\le\, I_1(\tau) +  I_2(\tau) +  
  I_3(\tau) \,\le\, c_4 |\beta|^{2/3}\,e^{-c_5 |\beta|^{1/3}\tau}\,,
\end{equation}
for some positive constants $c_4, c_5$. In fact, we can assume without
loss of generality that $c_4 \ge e^{c_5}$, in which case equation
\eqref{prelim} also holds for $0 \le \tau < |\beta|^{-1}$ because we
already know that $\|e^{-\tau H_{n,\beta}}\|_{Z \to Z} \le e^{-|n|\tau /2} \le
1$ for all $\tau \ge 0$. This proves estimate \eqref{semiest1}, and
exactly the same argument gives \eqref{semiest2} for $n = \pm 1$ if we
restrict the operator $H_{n,\beta}$ to the subspace $Z_0$ where 
$-H_{n,\beta} + 1$ is $m$-dissipative and inequality \eqref{LWZest3}
holds.
\end{proof}

\begin{rem}\label{rem:abscissa}
We emphasize that there is a lot of freedom in the choice of the
abscissa $x_0$ in the proof of Proposition~\ref{prop:semi1}. In fact,
the only real constraint is that the spectrum of $H_{n,\beta}$ be
entirely contained in the half-plane $\{\Re(z) > x_0\}$. In view of
estimates \eqref{LWZest2}, \eqref{LWZest3} this is certainly the case
if $x_0 = \kappa |\beta|^{1/3}/c_2$ for some $\kappa \in [0,1)$, in
which case a slight modification of the argument above gives the
bounds \eqref{semiest1}, \eqref{semiest2} with $c_5 = \kappa/c_2$ and
$c_4 > 0$ depending only on $\kappa$. However, we know from
\cite[Theorem~6.1]{LWZ} that all eigenvalues of the operator
$H_{n,\beta}$ (restricted to $Z_0$ if $n = \pm 1$) have real parts
larger than $c_3 |\beta|^{1/2}$ if $|\beta| \ge 1$, see also
Proposition~\ref{prop:LWZ2}. If we take $x_0$ such that $|\beta|^{1/3}/c_2 
< x_0 < c_3 |\beta|^{1/2}$, we obtain for $|n| \ge 2$ an estimate of 
the form
\begin{equation}\label{semiest3}
  \|e^{-\tau H_{n,\beta}}\|_{Z \to Z} \,\le\, C(x_0,|\beta|)
  \,e^{-x_0 \tau}\,, \qquad \tau \ge 0\,,
\end{equation}
which is clearly superior to \eqref{semiest1} for large times, as 
it implies that
\begin{equation}\label{semiest4}
  \limsup_{\tau \to \infty} \frac{1}{\tau} \log \|e^{-\tau H_{n,\beta}}\|_{Z \to Z} 
  \,\le\, - c_3 |\beta|^{1/2}\,.
\end{equation}
However, we have no control anymore on the constant $C(x_0,|\beta|)$, 
because the resolvent norm $\|(H_{n,\beta}-z)^{-1} \|$ can be extremely
large on the vertical line $\{\Re(z) = x_0\}$ which meets the 
pseudospectrum of $H_{n,\beta}$, see \cite[Lemma~1.3]{GGN} for a 
detailed discussion. In the applications to nonlinear stability
in Section~\ref{sec4}, we need to control the size of the perturbations 
not only in the limit $\tau \to \infty$, but also for intermediate times, 
and this is why we cannot use estimate \eqref{semiest3} for 
$x_0 > |\beta|^{1/3}/c_2$. 
\end{rem}

As a corollary of Proposition~\ref{prop:semi1}, we deduce the 
following decay estimate for the semigroup generated by the 
linearized operator $\cL - \alpha\Lambda$ in the subspace 
$X_\perp = \ker(\Lambda)^\perp$. 

\begin{prop}\label{prop:semi2} For any $\alpha \in \R$ with 
$|\alpha| \ge 1$, we have
\begin{equation}\label{semiest5}
  \|e^{\tau (\cL - \alpha\Lambda)}\|_{X_\perp \to X_\perp} \,\le\, 
  \min\Bigl(e^{-\tau}\,,\, c_7 |\alpha|^{2/3}\,e^{-c_5 
  |\alpha|^{1/3}\tau}\Bigr)\,, \qquad \tau \ge 0\,,
\end{equation}
where $c_7 = \max(c_4,e^2)$ and $c_4,c_5$ are as in 
Proposition~\ref{prop:semi1}.
\end{prop}

\begin{proof}
According to \eqref{Xdecomp}, any $w \in X_\perp$ can be represented in 
polar coordinates as
\[
  w \,=\, \sum_{n\neq 0} w_n(r) \,e^{in\theta}\,, \qquad 
  \|w\|_X^2 \,=\, 2\pi \sum_{n\neq 0} \|w_n\|_Z^2\,,
\]
where $w_{\pm 1} \in Z_0$ and $w_n \in Z$ for $|n| \ge 2$. In particular, 
we have by definition
\[
  \|e^{\tau (\cL - \alpha\Lambda)} w\|_X^2 \,=\, 2\pi \sum_{n\neq 0} 
  \|e^{\tau(\cL_n - \alpha\Lambda_n)} w_n\|_Z^2 \,=\,  2\pi \sum_{n\neq 0} 
  \|e^{-\tau H_{n,n\alpha}} w_n\|_Z^2\,, \qquad \tau \ge 0\,. 
\]
In view of Proposition~\ref{prop:semi1}, to prove \eqref{semiest5} we only 
need to verify that
\begin{equation}\label{semisup}
  \sup_{n\neq 0} \min\Bigl(e^{-\tau}\,,\, c_4 |n|^{2/3} |\alpha|^{2/3}
  \,e^{-c_5 |n|^{1/3} |\alpha|^{1/3}\tau}\Bigr) \,\le\, 
  \min\Bigl(e^{-\tau}\,,\, c_7 |\alpha|^{2/3} \,e^{-c_5 |\alpha|^{1/3}\tau}\Bigr)\,, 
\end{equation}
for any $\tau \ge 0$ and any $\alpha \in \R$ with $|\alpha| \ge 1$. 

Let $T = c_5 |\alpha|^{1/3}\tau$. If $T \ge 2$, the quantity
$|n|^{2/3} \,e^{-|n|^{1/3}T}$ reaches its maximum (over all nonzero
integers) when $|n| = 1$, and \eqref{semisup} follows immediately 
since $c_7 \ge c_4$. If $0 \le T < 2$, we observe that
\[
  e^{-\tau} \,\le\, 1 \,\le\, |\alpha|^{2/3} \,e^{2-T} \,\le\, c_7 
  |\alpha|^{2/3}\,e^{-c_5 |\alpha|^{1/3}\tau}\,,
\]
because $|\alpha| \ge 1$ and $c_7 \ge e^2$, hence \eqref{semisup}
holds in that case too. 
\end{proof}

\begin{rem}\label{numrem}
It is also possible to establish \eqref{semiest5} using the inverse
Laplace representation for the semigroup $e^{\tau (\cL - \alpha\Lambda)}$ 
in $X_\perp$ and the resolvent estimate given by Proposition~\ref{prop:LWZ}.
In that alternative approach, one needs to locate the numerical range
of the operator $-\cL + \alpha\Lambda$ in $X_\perp$, in order to
choose an appropriate integration path. From \eqref{numrange} 
it is easy to deduce that
\[
  \cN(-\cL_\perp +\alpha\Lambda_\perp) \,\subset\, \bigl\{z \in \C 
  \,\big|\, \Re(z) \ge 1\,,~|\Im(z)| \le 2c_6|\alpha| \Re(z) \bigr\}\,,
\]
but this ``sectorial'' estimate can be improved as is shown by the 
following result.
\end{rem}

\begin{lem}\label{lem:numrange}
There exist positive constants $c_8$, $c_9$ such that
\[
  \cN(-\cL_\perp +\alpha\Lambda_\perp) \,\subset\, \bigl\{z \in \C 
  \,\big|\, \Re(z) \ge 1\,,~|\Im(z)| \le |\alpha|(c_8 \Re(z)^{1/2}
  + c_9)\bigr\}\,.
\] 
\end{lem}

\begin{proof}
Let $w \in D(\cL) \subset X$. If $w \in X_1$, inequality \eqref{E2} 
in Lemma~\ref{lem:E} below shows that
\begin{equation}\label{cLest}
  \langle w, \cL w \rangle_X \,\le\, -\frac14 \|\nabla w\|_X^2  
  -\frac{1}{64} \|\xi w\|_X^2 - \frac18 \|w\|_X^2\,.
\end{equation}
To bound the quantity $\langle w, \Lambda w \rangle_X$, we decompose
$\Lambda = \Lamad + \Lamnl$ as in \eqref{Lambdadef}. We first observe 
that
\begin{equation}\label{Lamad}
  |\langle w, \Lamad w \rangle_X| \,=\, \Bigl|\int_{\R^2} G^{-1} \bar w 
  v^G \cdot \nabla w \dd\xi\Bigr| \,\le\, \|v^G\|_{L^\infty} 
  \|w\|_X \|\nabla w\|_X\,.
\end{equation}
On the other hand, since $\nabla G = -\frac12 \xi G$, we have
\[
  \langle w, \Lamnl w \rangle_X \,=\, \int_{\R^2} G^{-1} \bar w
  (K_{BS}*w)\cdot \nabla G \dd \xi \,=\, -\frac12 \int_{\R^2} \bar w
  (K_{BS}*w)\cdot \xi \dd \xi\,. 
\]
In view of Lemma~\ref{lem:BS1} below, we have $\|K_{BS}*w\|_{L^4}
\le C \|w\|_{L^{4/3}}$, hence
\begin{equation}\label{Lamnl}
  |\langle w, \Lamnl w \rangle_X| \,\le\, \frac12 \|K_{BS}*w\|_{L^4}
  \||\xi| w\|_{L^{4/3}} \,\le\, C \|w\|_{L^{4/3}} \||\xi| w\|_{L^{4/3}}
 \,\le\, C \|w\|_X^2\,. 
\end{equation}

Now, we fix $\alpha \in \R$, and we assume that $w \in D(\cL) \cap 
X_\perp$ is normalized so that $\|w\|_X = 1$. We consider the 
complex number
\[
  z \,=\, \langle w , (-\cL +\alpha\Lambda)w\rangle_X \,\in\, 
  \cN(-\cL_\perp +\alpha\Lambda_\perp)\,,
\] 
which satisfies $\Re(z) = -\langle w, \cL w \rangle_X \ge 1$ and 
$\Im(z) = \alpha \langle w, \Lambda w \rangle_X$. We know 
from \eqref{cLest}, \eqref{Lamad}, \eqref{Lamnl} that
\[
  \Re(z) \ge \frac14 \|\nabla w\|_X^2 + \frac18, \qquad 
  |\Im(z)| \,\le\, C|\alpha|\,(\|\nabla w\|_X + 1)\,,
\]
and this implies that $|\Im(z)| \le |\alpha|(c_8 \Re(z)^{1/2} + c_9)$
for some $c_8, c_9 > 0$. 
\end{proof}

The semigroup estimates given by Propositions~\ref{prop:semi1} and
\ref{prop:semi2} are not sufficient by themselves to prove 
Theorem~\ref{thm:main}, mainly because the nonlinear term in equation 
\eqref{tildeweq} involves a derivative. Using the inverse Laplace 
formula, it is possible to obtain accurate bounds on the first order 
derivative $\nabla e^{\tau(\cL-\alpha\Lambda)}$ for large times, but the 
problem is that we also need an estimate for short times that
does not blow up in the large circulation limit $|\alpha| \to 
\infty$. For the semigroup itself (without derivative), such 
an estimate was available without any pain thanks to the 
dissipativity properties of the generator $\cL-\alpha\Lambda$. 
In a similar spirit, the following elementary result will allow 
us in Section~\ref{sec4} to control the nonlinearities for short
times without loosing any power of the circulation parameter 
$|\alpha|$.

\begin{lem}\label{lem:dissip}
There exists a constant $C_0 > 0$ such that, for any $\alpha \in \R$, 
any $T > 0$, and any $f \in C^0([0,T],X^2)$, we have the 
estimate
\begin{equation}\label{dissipest}
  \Bigl\|\int_0^\tau e^{(\tau-s)(\cL-\alpha\Lambda)} \,\div f(s)
  \dd s\Bigr\|_X^2 \,\le\, C_0 \int_0^\tau \|f(s)\|_X^2 \dd s\,,
  \qquad \tau \in [0,T]\,.
\end{equation}
\end{lem}

\begin{proof}
Denote $w(\tau) = \int_0^\tau e^{(\tau-s)(\cL-\alpha\Lambda)} \,\div f(s)
\dd s$ for $\tau \in [0,T]$. Then $w \in C^0([0,T],X)$ is 
the unique solution of the linear evolution equation
\begin{equation}\label{wlineq}
  \partial_\tau w \,=\, (\cL - \alpha\Lambda) w + \div f\,, 
  \qquad \tau \in [0,T]\,,
\end{equation}
with initial data $w(0) = 0$. Moreover it is clear that 
$w(\tau) \in X_0$ for any $\tau \in [0,T]$, because the
source term in \eqref{wlineq} has zero mean over $\R^2$. 
A direct calculation, using the skew-symmetry of the 
operator $\Lambda$ in $X$, shows that
\begin{equation}\label{wlinest}
  \frac12\frac{\D}{\D \tau} \|w(\tau)\|_X^2 \,=\, \langle w, \cL w 
  \rangle_X + \langle w, \div f \rangle_X\,. 
\end{equation}
As $w \in X_0$, we know from Lemma~\ref{lem:E} that
\begin{equation}\label{cLest2}
  \langle w, \cL w \rangle_X \,\le\, -\frac16 \|\nabla w\|_X^2  
  -\frac{1}{96} \|\xi w\|_X^2 - \frac{1}{12} \|w\|_X^2\,.
\end{equation}
On the other hand, integrating by parts and using the fact that
$\nabla G^{-1} = \frac12 \xi G^{-1}$, we obtain
\begin{equation}\label{fest}
  \langle w, \div f \rangle_X \,=\, -\int_{\R^2} G^{-1} 
  f \cdot \Bigl(\nabla w + \frac{\xi}2 w\Bigr)\dd\xi
  \,\le\, \epsilon \bigl(\|\nabla w\|_X^2 + \|\xi w\|_X^2\bigr)
  + \frac{C}{\epsilon}\,\|f\|_X^2\,.
\end{equation}
If we take $\epsilon > 0$ small enough, we can combine \eqref{wlinest}, 
\eqref{cLest2} and \eqref{fest} to obtain the differential 
inequality
\[
   \frac{\D}{\D \tau} \|w(\tau)\|_X^2 \,\le\, -\frac16 \|w(\tau)\|_X^2 
   + C_0  \|f(\tau)\|_X^2 \,\le\, C_0  \|f(\tau)\|_X^2\,, \qquad 
   \tau \in [0,T]\,,
\]
for some positive constant $C_0$, and \eqref{dissipest} follows 
upon integrating over $\tau$.
\end{proof}

\section{Nonlinear stability and relaxation to axisymmetry}
\label{sec4}

Equipped with the semigroups estimates derived in Section~\ref{sec3}, 
we now study the nonlinear stability of the equilibrium $w = 
\alpha G$ for the rescaled vorticity equation \eqref{weq}. 
We assume that the circulation parameter $\alpha \in \R$ satisfies 
$|\alpha| \ge \alpha_0$, where $\alpha_0 > 0$ is large enough and 
will be determined later. Given any $T > 0$, we consider a 
solution $w \in C^0([0,T],X)$ of \eqref{weq} with initial data 
$w_0 = \alpha G + \tilde w_0$, where $\tilde w_0 \in X$ satisfies 
$\|\tilde w_0\|_X \le C_3 |\alpha|$ for some small constant $C_3 > 0$.

\subsection{Preliminaries}\label{subsec4.1}

We start with two elementary observations which allow us to concentrate
on the situation where the initial perturbation $\tilde w_0$ 
belongs to the subspace $X_1$ defined by \eqref{X1def}. 

\medskip\noindent {\bf Observation 1\:} Without loss of generality, we 
can assume that $\tilde w_0 \in X_0$, where $X_0 \subset X$ is the subspace
defined in \eqref{X0def}. Indeed, if $\tilde \alpha = \int_{\R^2}
\tilde w_0 \dd\xi \neq 0$, we decompose
\[
  w_0 \,=\, \hat \alpha G + \hat w_0\,, \qquad \hbox{where}\quad
  \hat \alpha \,=\, \alpha + \tilde \alpha \quad \hbox{and}
  \quad \hat w_0 \,=\, \tilde w_0 - \tilde \alpha G\,.
\]
Then $\hat w_0 \in X_0$ by construction, and since $|\tilde \alpha| \le 
C \|\tilde w_0\|_X \le C C_3 |\alpha|$ we can assume that $|\hat\alpha| 
\ge |\alpha|/2 \ge \alpha_0/2$ by taking the constant $C_3$
sufficiently small. The problem is thus reduced to the stability
analysis of the modified vortex $\hat\alpha G$ with respect to
perturbations $\hat w_0 \in X_0$, and we still have $\|\hat w_0\|_X
\le C_3' |\hat \alpha|$ for some small constant $C_3'$. 

\medskip\noindent {\bf Observation 2\:} Without loss of generality, we can 
further assume that $\tilde w_0 \in X_1$, where $X_1 \subset X$ 
is defined in \eqref{X1def}. Indeed, let $w$ be the solution 
of \eqref{weq} with initial data $\alpha G + \tilde w_0$, where
$\alpha \neq 0$ and $\tilde w_0 \in X_0$ satisfies $\|\tilde w_0\|_X 
\le C_3|\alpha|$. If $\tilde w_0 \notin X_1$, we introduce the 
first order moment
\[
  \eta \,=\, \frac{1}{\alpha} \int_{\R^2} \xi \,w_0(\xi)\dd\xi 
  \,=\, \frac{1}{\alpha} \int_{\R^2} \xi \,\tilde w_0(\xi)\dd\xi 
  \,\in\,\R^2\,,
\]
and we consider the modified vorticity $\hat w$ and velocity 
$\hat v$ defined by
\begin{equation}\label{etamodif}
  \hat w(\xi,\tau) \,=\, w(\xi + \eta\,e^{-\tau/2},\tau)\,, \qquad
  \hat v(\xi,\tau) \,=\, v(\xi + \eta\,e^{-\tau/2},\tau)\,. 
\end{equation}
It is straightforward to verify that the new functions $\hat w$, 
$\hat v$ satisfy the same equation \eqref{weq}, namely 
$\partial_\tau \hat w + \hat v \cdot \nabla \hat w = \cL \hat w$. 
In addition, the explicit expression
\[
  \hat w(\xi,0) - \alpha G(\xi) \,=\, \alpha\bigl(G(\xi+\eta) -
  G(\xi)\bigr) + \tilde w_0(\xi+\eta)\,, \qquad \xi \in \R^2\,,
\]
reveals that $\hat w(\cdot,0) - \alpha G \in X_1$ and $\|\hat
w(\cdot,0) - \alpha G\|_X \le C \|\tilde w_0\|_X \le CC_3
|\alpha|$. Thus the change of variables \eqref{etamodif} allows us to
reduce the stability analysis to perturbations in the subspace
$X_1$. In terms of the original variables, that transformation is
equivalent to choosing the parameter $x_0$ in \eqref{wvdef} to be the
{\em center of vorticity} of the distribution $\omega(\cdot,t)$, which is
well defined as soon as $\alpha \neq 0$ and preserved under 
the evolution defined by \eqref{omeq}. 

\subsection{Decomposition of the perturbations}\label{subsec4.2}

Taking into account the observations above, we assume henceforth 
that $\tilde w \in C^0([0,T],X_1)$ is a solution of \eqref{tildeweq}
with initial data $\tilde w_0 \in X_1$ satisfying $\|\tilde w_0\|_X 
\le C_3 |\alpha|$. According to Section~\ref{subsec2.1}, the 
perturbation space $X_1$ can decomposed as $X_1 = X_r \oplus X_\perp$, 
where $X_r = X_0 \cap Y_0$ is the subset of $X$ consisting of all radially
symmetric functions with zero average, and $X_\perp = \ker(\Lambda)^\perp$
is the orthogonal complement of $\ker(\Lambda)$ in $X$. We thus 
decompose the perturbed vorticity as 
\begin{equation}\label{wdecomp}
  \tilde w \,=\, \tilde w_r + \tilde w_\perp \,=\, P_r \tilde w + 
  P_\perp \tilde w\,,
\end{equation}
where $P_r = 1 - P_\perp$ is the orthogonal projection of $X_1$ onto $X_r$. 
If we introduce polar coordinates $(r,\theta)$ such that $\xi = (r\cos\theta,
r\sin\theta)$, we have the explicit expression
\begin{equation}\label{Prdef}
   (P_r \tilde w)(r) \,=\, \frac{1}{2\pi} \int_{-\pi}^\pi \tilde w(r,\theta)
   \dd\theta\,, \qquad r  > 0\,.
\end{equation}
The corresponding decomposition of the velocity is
\begin{equation}\label{vdecomp}
  \tilde v \,=\,  \tilde v_r + \tilde v_\perp \,=\, 
  K_{BS} * \tilde w_r + K_{BS} * \tilde w_\perp\,.
\end{equation}
Denoting $e_r = (\cos\theta,\sin\theta)$, $e_\theta = (-\sin\theta,
\cos\theta)$, we have the following elementary result\:

\begin{lem}\label{lem:decomp}
With the definitions \eqref{wdecomp}, \eqref{vdecomp}, 
we have $\tilde v_r \cdot \nabla \tilde w_r = 0$ and 
\begin{equation}\label{nonproj}
  P_r (\tilde v \cdot \nabla\tilde w) =  P_r (\tilde v_\perp \cdot 
  \nabla \tilde w_\perp) = \div Z[\tilde w,\tilde v]\,,
\end{equation}
where $Z$ is the vector field defined by
\begin{equation}\label{Zdef}
  Z[\tilde w,\tilde v] \,=\, P_r \bigl((\tilde v_\perp\cdot e_r)\tilde w_\perp
  \bigr) e_r\,. 
\end{equation}
\end{lem}

\begin{proof}
Since $\tilde w_r$ is radially symmetric, the associated velocity 
$\tilde v_r = K_{BS}*\tilde w_r$ is purely azimuthal, namely 
$\tilde v_r \cdot e_r = 0$, and this implies that $\tilde v_r \cdot 
\nabla \tilde w_r = 0$. On the other hand, as $\tilde v$ is 
divergence free, we have
\[
  \tilde v \cdot \nabla\tilde w \,=\, \div\bigl(\tilde v \tilde w\bigr)
  \,=\, \frac{1}{r}\partial_r \bigl(r (\tilde v \cdot e_r) \tilde w\bigr)
  + \frac{1}{r}\partial_\theta \bigl((\tilde v \cdot e_\theta) \tilde 
  w\bigr)\,.
\]
If we apply the projection $P_r$, the last term in the right-hand
side gives no contribution, and in the first term we have $\tilde v
\cdot e_r = \tilde v_\perp \cdot e_r$, as already observed. That 
quantity has zero average over the angular variable $\theta$, and 
this implies that $P_r\bigl((\tilde v_\perp \cdot e_r)\tilde w\bigr) = 
P_r\bigl((\tilde v_\perp \cdot e_r)\tilde w_\perp\bigr)$. Summarizing, 
we have shown that \eqref{nonproj} holds if $Z$ is defined by 
\eqref{Zdef}. 
\end{proof}

In view of Lemma~\ref{lem:decomp}, the perturbation equation 
\eqref{tildeweq} is equivalent to the coupled system
\begin{equation}
\begin{aligned}\label{wrperp}
  \partial_\tau \tilde w_r + P_r (\tilde v_\perp \cdot \nabla \tilde w_\perp) 
  \,&=\, \cL \tilde w_r\,, \\ 
  \partial_\tau \tilde w_\perp + \tilde v_r \cdot \nabla \tilde w_\perp 
  + \tilde v_\perp \cdot \nabla \tilde w_r + P_\perp (\tilde v_\perp \cdot 
  \nabla \tilde w_\perp) \,&=\, (\cL -\alpha\Lambda) \tilde w_\perp\,,
\end{aligned}
\end{equation}
which is the starting point of our analysis. The integrated version 
of \eqref{wrperp} is written in the form
\begin{equation}
\begin{aligned}\label{wrpint}
  \tilde w_r(\tau) \,&=\, e^{\tau\cL}\tilde w_r(0) - \int_0^\tau e^{(\tau-s)\cL} 
  \div Z[\tilde w(s),\tilde v(s)] \dd s\,, \\
  \tilde w_\perp(\tau) \,&=\, e^{\tau(\cL-\alpha\Lambda)} \tilde w_\perp(0) - 
  \int_0^\tau e^{(\tau-s)(\cL-\alpha\Lambda)} \div N[\tilde w(s),\tilde v(s)]\dd s\,,
\end{aligned}
\end{equation}
where
\begin{equation}\label{Ndef}
  N[\tilde w,\tilde v] \,=\, \tilde v_r\,\tilde w_\perp + \tilde v_\perp\,
  \tilde w_r + \tilde v_\perp\,\tilde w_\perp - Z[\tilde w,\tilde v]\,.
\end{equation}
According to Proposition~\ref{prop:semi2}, the semigroups in 
\eqref{wrpint} satisfy, for all $\tau \ge 0$, 
\begin{equation}\label{sgbounds}
  \|e^{\tau\cL} \tilde w_r\|_X \,\le\, e^{-\tau}\|\tilde w_r\|_X\,, \qquad 
  \|e^{\tau(\cL-\alpha\Lambda)} \tilde w_\perp\|_X \,\le\, \min(e^{-\tau},
  e^{-\mu(\tau-\tau_0)})\|\tilde w_\perp\|_X\,,
\end{equation}
where
\begin{equation}\label{mutaudef}
  \mu \,=\, c_5 |\alpha|^{1/3}\,, \qquad \tau_0 \,=\, \frac{\log(c_7
  |\alpha|^{2/3})}{\mu}\,.
\end{equation}
In what follows we assume that the constant $\alpha_0 \ge 2$ is large 
enough so that, if $|\alpha| \ge \alpha_0$, the quantities defined 
in \eqref{mutaudef} satisfy $\mu \ge 1$, $0 < \tau_0 \le 1$, 
and $\mu \tau_0 \ge 3$.  

\subsection{Nonlinear estimates}\label{subsec4.3}

Keeping the same notations as in the previous section, our goal is to 
control the solution of \eqref{wrpint} using the following norm
\begin{equation}\label{defnorm}
  \cM \,=\, \cM(T) \,=\, \sup_{0 \le \tau \le T} \Bigl(\|\tilde w_r(\tau)\|_X 
  + e^{\tau/\tau_0} \|\tilde w_\perp(\tau)\|_X\Bigr)\,.  
\end{equation}
The main technical result of this section is\:

\begin{lem}\label{lem:nonlin} There exist positive constants $C_4$, 
$C_5$ such that, for any $\alpha \in \R$ with $|\alpha| \ge \alpha_0$, 
any $T > 0$, and any solution $(\tilde w_r,\tilde w_\perp) \in C^0([0,T], 
X_r \oplus X_\perp)$ of \eqref{wrpint} satisfying $\cM \le |\alpha|$, 
the following estimate holds\:
\begin{equation}\label{nonlinest}
  \cM \,\le\, C_4 \|\tilde w_0\|_X + C_5 \bigl(\tau_0 \log|\alpha|\bigr)^{1/2}
  \cM^2 + C_5\bigl(\tau_0 \log_+\cM^{-1}\bigr)^{1/2}\cM^2\,,
\end{equation}
where $\tilde w_0 = \tilde w_r(0) + \tilde w_\perp(0)$. Here 
$\log_+(x) = \max\bigl(\log(x),0\bigr)$ for any $x > 0$. 
\end{lem}

\begin{proof} We first establish a simple a priori estimate that will 
be useful later. By construction, the function $w(\tau) = \alpha G + 
\tilde w_r(\tau) + \tilde w_\perp(\tau)$ is a solution of the rescaled 
vorticity equation \eqref{weq}, hence of the integral equation \eqref{wint} 
below. We also know that $\|w(\tau)\|_X \le C|\alpha|$ for all 
$\tau \in [0,T]$, because $\cM \le |\alpha|$ by assumption. Our 
goal is to control the norm $\|G^{-1/2} w(\tau)\|_{L^3}$ for any 
$\tau \in [0,T]$ using the representation \eqref{wint} and the 
semigroup bounds recalled in Section~\ref{subsec5.1}. Applying
estimate \eqref{LpLq} with $q = 3$, $p = 2$, inequality
\eqref{LpLqder} with $q = 3$, $p = 4/3$, and finally estimate 
\eqref{Gv1w2} with $w_1 = w_2 = w$, we obtain for any $\tau \in (0,T]$\: 
\begin{align*}
  \|G^{-1/2} w(\tau)\|_{L^3} \,&\le\, \frac{C}{a(\tau)^{1/6}}\,\|w(0)\|_X
  + \int_0^\tau \frac{C e^{-(\tau-s)/2}}{a(\tau{-}s)^{11/12}} \,\|G^{-1/2} 
  v(s) w(s) \|_{L^{4/3}}\dd s\,, \\
  \,&\le\, \frac{C}{a(\tau)^{1/6}}\,\|w(0)\|_X
  + \int_0^\tau \frac{C e^{-(\tau-s)/2}}{a(\tau{-}s)^{11/12}} 
  \,\|w(s)\|_X^2 \dd s\,, 
\end{align*}
where $a(\tau) = 1 - e^{-\tau}$. As $|\alpha| \ge 1$, it readily 
follows that 
\begin{equation}\label{apriori}
  \|G^{-1/2} \tilde w_r(\tau)\|_{L^3} + \|G^{-1/2} \tilde w_\perp(\tau)\|_{L^3} 
  \,\le\, \frac{C_6|\alpha|^2}{a(\tau)^{1/6}}\,, \qquad 0 < \tau \le T\,,
\end{equation}
for some positive constant $C_6$. 

We now focus on the proof of \eqref{nonlinest}. The overall strategy
is to estimate the right-hand side of system \eqref{wrpint} using the
semigroup bounds \eqref{sgbounds} and the integral estimate 
\eqref{dissipest}. We start with the equation satisfied by the
radially symmetric component $\tilde w_r$, which is simpler to 
handle. We know that $\| e^{\tau\cL}\tilde w_r(0)\|_X \le \|\tilde w_r(0)\|_X$, 
and using Lemma~\ref{lem:dissip} together with definition \eqref{Zdef} 
we obtain
\begin{align*}
  \Bigl\|\int_0^\tau e^{(\tau-s)\cL} \div Z[\tilde w(s),\tilde v(s)]\dd s\Bigr\|_X^2
  \,&\le\, C_0 \int_0^\tau \|\tilde w_\perp(s) \tilde v_\perp(s)\|_X^2 \dd s \\
  \,&\le\, C_0 \int_0^\tau \|\tilde w_\perp(s)\|_X^2 \,\|\tilde v_\perp(s)
  \|_{L^\infty}^2 \dd s\,.
\end{align*}
To control the $L^\infty$ norm of the velocity field $\tilde v_\perp(s)$, 
we apply the results of Section~\ref{subsec5.2}. Since $\|\tilde 
w_\perp(s)\|_{L^1\cap L^2} \le C\|\tilde w_\perp(s)\|_X$ and $\|\tilde 
w_\perp(s)\|_X \le \cM$, it follows from Lemma~\ref{lem:BS3} that
\begin{align*}
  \|\tilde v_\perp(s)\|_{L^\infty}^2 \,&\le\, C\|\tilde w_\perp(s)\|_X^2 
  \Bigl(1 + \log_+ \frac{\|\tilde w_\perp(s)\|_{L^3}}{\|\tilde w_\perp(s)\|_X}
  \Bigr) \\ \,&\le\, C\,\cM^2 \Bigl(1 + \log_+ \frac{\|\tilde w_\perp(s)\|_{L^3}}{
  \cM}\Bigr) \,\le\, C\cM^2 \Bigl(1 + \log_+ \frac{C_6|\alpha|^2}{a(s)^{1/6}
  \cM}\Bigr)\,,
\end{align*}
where in the second inequality we used the fact that the map 
$a \mapsto a^2\bigl(1 + \log_+(b/a)\bigr)$ is strictly increasing 
over $\R_+$ for any $b > 0$, and in the last inequality we 
invoked the a priori estimate \eqref{apriori}. Observing that 
$\log_+(ab) \le \log_+(a) + \log_+(b)$ and $\|\tilde w_\perp(s)\|_X \le 
\cM e^{-s/\tau_0}$, we obtain after integrating over time
\[
  \int_0^\tau \|\tilde w_\perp(s)\|_X^2 \,\|\tilde v_\perp(s)
  \|_{L^\infty}^2 \dd s \,\le\,  C \cM^4 \tau_0 \Bigl(1+\log\frac{
  |\alpha|}{\tau_0} + \log_+\frac{1}{\cM}\Bigr)\,,
\]
where $\log(|\alpha|/\tau_0) \le C\log|\alpha|$ in view of definition
\eqref{mutaudef}. Altogether we have shown that
\begin{equation}\label{wrest}
  \sup_{0 \le \tau \le T}\|\tilde w_r(\tau)\|_X \,\le\, \|\tilde w_r(0)\|_X + 
  C \bigl(\tau_0 \log|\alpha|\bigr)^{1/2} \cM^2 + C\bigl(\tau_0 
  \log_+\cM^{-1}\bigr)^{1/2}\cM^2\,.
\end{equation}

We next consider the second equation in \eqref{wrpint}. When $\tau \le 
\tau_0$ the spatial weight $e^{\tau/\tau_0}$ does not play any role
in definition \eqref{defnorm}, so repeating the arguments above 
we obtain an estimate of the form \eqref{wrest} for $\|\tilde 
w_\perp(\tau)\|_X$ if $\tau \in [0,\tau_0]$. In the rest of the 
proof, we thus assume that $\tau > \tau_0$, and we decompose $\tau = 
N\tau_0 + \tau_1$ where $N \in \N$ and $\tau_0 < \tau_1 \le 2\tau_0$. 
Since $\mu \tau_0 \ge 1$, we obviously have
\begin{equation}\label{wperp1}
  \|e^{\tau(\cL-\alpha\Lambda)} \tilde w_\perp(0)\|_X \,\le\, e^{-\mu(\tau-\tau_0)}
  \|\tilde w_\perp(0)\|_X \,\le\, e^{1-\tau/\tau_0} \|\tilde w_\perp(0)\|_X\,.
\end{equation}
On the other hand, the integral term can be decomposed 
in the following way
\begin{equation}\label{intdecomp}
  \int_0^\tau e^{(\tau-s)(\cL-\alpha\Lambda)} \div N[\tilde w(s),\tilde v(s)]\dd s 
  \,=\, I_0(\tau) + \sum_{k=1}^N J_k(\tau)\,,
\end{equation}
where
\begin{align*}
  I_0(\tau) \,&=\, \int_{N\tau_0}^\tau e^{(\tau-s)(\cL-\alpha\Lambda)} \div 
  N[\tilde w(s),\tilde v(s)] \dd s\,, \\
  J_k(\tau) \,&=\, e^{(\tau-k\tau_0)(\cL-\alpha\Lambda)} \int_{(k-1)\tau_0}^{k\tau_0} 
  e^{(k\tau_0-s)(\cL-\alpha\Lambda)} \div N[\tilde w(s),\tilde v(s)]\dd s\,.
\end{align*}

To control the nonlinear term $N[\tilde w,\tilde v]$ defined in 
\eqref{Ndef}, we again apply the results of Section~\ref{subsec5.2}. 
We first observe that $\|\tilde v_r\,\tilde w_\perp + \tilde v_\perp\,
\tilde w_\perp - Z[\tilde w,\tilde v]\|_X \le C \bigl(\|\tilde v_r\|_{L^\infty} 
+ \|\tilde v_\perp\|_{L^\infty}\bigr) \|w_\perp\|_X$, and we use 
Lemma~\ref{lem:BS3} to obtain
\begin{align*}
  \|\tilde v_r\|_{L^\infty}^2 \,&\le\, C\,\|\tilde w_r\|_X^2 \Bigl(1 + 
  \log_+ \frac{\|\tilde w_r\|_{L^3}}{\|\tilde w_r\|_X}\Bigr) \,\le\, 
  C \cM^2 \Bigl(1 + \log_+ \frac{\|\tilde w_r\|_{L^3}}{\cM}\Bigr)\,, \\
  \|\tilde v_\perp\|_{L^\infty}^2 \,&\le\, C\,\|\tilde w_\perp\|_X^2 
  \Bigl(1 + \log_+ \frac{\|\tilde w_\perp\|_{L^3}}{\|\tilde w_\perp\|_X}
  \Bigr) \,\le\, C \cM^2 \Bigl(1 + \log_+ \frac{\|\tilde w_\perp\|_{L^3}}{
  \cM}\Bigr)\,.
\end{align*}
The last term $\tilde v_\perp \tilde w_r$ in $N[\tilde w,\tilde v]$ is 
estimated directly by applying Lemma~\ref{lem:BS4} with $\omega_1 = 
G^{-1/2}\tilde w_r$ and $\omega_2 = \tilde w_\perp$. This gives
\begin{align*}
  \|\tilde v_\perp \tilde w_r\|_X^2 \,&\le\, C\,\|\tilde w_r\|_X^2 
  \,\|\tilde w_\perp\|_X^2 \Bigl(1 + \log_+ \frac{\|G^{-1/2}\tilde w_r\|_{L^3}}{
  \|\tilde w_r\|_X}\Bigl) \\
  \,&\le\, C \cM^2 \,\|\tilde w_\perp\|_X^2 \Bigl(1 + \log_+ \frac{
  \|G^{-1/2}\tilde w_r\|_{L^3}}{\cM}\Bigl)\,.
\end{align*}
Summarizing both cases and using the a priori bound \eqref{apriori}, 
we arrive at
\begin{equation}\label{Nest}
  \|N[\tilde w(s),\tilde v(s)]\|_X^2 \,\le\, C \cM^4\,e^{-2s/\tau_0} 
  \Bigl(1 + \log_+ \frac{C_6|\alpha|^2}{a(s)^{1/6}\cM}\Bigl)\,,
  \qquad 0 < s \le T\,.
\end{equation}

We now estimate the various terms in \eqref{intdecomp}. Using 
\eqref{Nest} and Lemma~\ref{lem:dissip}, we first obtain
\[
  \|I_0(\tau)\|_X^2 \,\le\, C_0 \int_{N\tau_0}^\tau  \|N[\tilde w(s),
  \tilde v(s)]\|_X^2 \dd s \,\le\, C \cM^4 \tau_0 \,e^{-2N} 
  \Bigl(1+\log\frac{|\alpha|}{\tau_0} + \log_+\frac{1}{\cM}\Bigr)\,.
\]
Similarly, for $k \in \{1,\dots,N\}$, we find
\begin{align*}
  \|J_k(\tau)\|_X^2 \,&\le\, C_0\,e^{-2\mu(\tau-(k+1)\tau_0)} \int_{(k-1)\tau_0}^{
  k\tau_0} \|N[\tilde w(s),\tilde v(s)]\|_X^2 \dd s \\
  \,&\le\, C\cM^4 \tau_0 \,e^{-2\mu(\tau-(k+1)\tau_0)} \,e^{-2k}
  \Bigl(1+\log\frac{|\alpha|}{\tau_0} + \log_+\frac{1}{\cM}\Bigr)\,.
\end{align*}
We know that $e^{-N} \le e^{2-\tau/\tau_0}$ by definition of $N$, and 
since $\mu \tau_0 \ge 3$ by assumption we also have
\[
  e^{-\mu(\tau-\tau_0)}\sum_{k=1}^N e^{k(\mu\tau_0-1)} \,\le\, 2 e^{-\mu(\tau-\tau_0)}
  \,e^{N(\mu\tau_0-1)} \,\le\, 2\,e^{-N} \,\le\, 2\,e^{2-\tau/\tau_0}\,.
\]
Therefore the estimates above imply that
\begin{equation}\label{IJest}
  \|I_0(\tau)\|_X \,+\, \sum_{k=1}^N \|J_k(\tau)\|_X \,\le\, C\cM \tau_0^{1/2} 
  \,e^{-\tau/\tau_0}\Bigl(1+\log\frac{|\alpha|}{\tau_0} + 
  \log_+\frac{1}{\cM}\Bigr)^{1/2}\,.  
\end{equation}
Combining \eqref{wperp1}, \eqref{intdecomp}, and \eqref{IJest}, 
we thus obtain
\begin{equation}\label{wperpest}
  \sup_{0 \le \tau \le T} e^{\tau/\tau_0} \|\tilde w_\perp(\tau)\|_X \,\le\, 
  e \|\tilde w_\perp(0)\|_X + C \bigl(\tau_0 \log|\alpha|\bigr)^{1/2} 
  \cM^2 + C\bigl(\tau_0 \log_+\cM^{-1}\bigr)^{1/2}\cM^2\,.
\end{equation}
Estimate \eqref{nonlinest} is now a direct consequence of \eqref{wrest} 
and \eqref{wperpest}. 
\end{proof}

\begin{rem}\label{cor:nonlin}
Using the definition of $\tau_0$ in \eqref{mutaudef}, we see that, 
if $|\alpha|$ is large enough, estimate \eqref{nonlinest} can 
be written in the alternative form
\begin{equation}\label{nonlinest2}
  \cM(T) \,\le\, C_4 \|\tilde w_0\|_X + \frac{C_7\cM(T)^2}{|\alpha|^{1/6}}
  \Bigl(\log|\alpha| + \log_+\frac{1}{\cM(T)}\Bigr)\,,
\end{equation}
for some universal constants $C_4 \ge 1$ and $C_7 > 0$. 
\end{rem}

\subsection{End of the proof of Theorem~\ref{thm:main}.}

Choose $C_1 > 0$ such that $8 C_1 C_4 C_7 < 1$, and fix $\alpha_0 \ge 2$ 
large enough so that, whenever $|\alpha| \ge \alpha_0$\:\\[1mm]
i) The quantities $\mu$ and $\tau_0$ defined in \eqref{mutaudef}
satisfy $\mu \ge 1$, $0 < \tau_0 \le 1$, and $\mu \tau_0 \ge 3$.  \\[1mm]
ii) Estimate \eqref{nonlinest2} in Remark~\ref{cor:nonlin} is valid. \\[1mm]
iii) The following inequalities hold\: $C_1 |\alpha|^{1/6} \le C_3 
|\alpha| \log|\alpha|$, and $4 C_7 \le |\alpha|^{1/6}$.  

\medskip\noindent
Given $\alpha \in \R$ with $|\alpha| \ge \alpha_0$, we consider initial 
perturbations $\tilde w_0 \in X_1$ such that
\begin{equation}\label{initbound}
   \|\tilde w_0\|_X \,\le\, \frac{C_1 |\alpha|^{1/6}}{\log|\alpha|}\,,
   \qquad \hbox{hence also} \quad \|\tilde w_0\|_X \,\le\, C_3 |\alpha|\,.  
\end{equation}
By continuity, the solution of \eqref{tildeweq} with initial data
$\tilde w_0$ satisfies, for $T > 0$ sufficiently small,  
\begin{equation}\label{Mcontract}
   \cM(T) \,\le\, 2C_4\|\tilde w_0\|_X \,<\, \frac{1}{4 C_7}
   \,\frac{|\alpha|^{1/6}}{\log|\alpha|}\,,
\end{equation}
where $\cM(T)$ is defined in \eqref{defnorm}. But as long as 
\eqref{Mcontract} holds, we have by construction
\[
  \frac{C_7\cM(T)^2}{|\alpha|^{1/6}}\log|\alpha| \,<\, \frac{\cM(T)}{4}\,, 
  \qquad \hbox{and}\quad \frac{C_7\cM(T)^2}{|\alpha|^{1/6}}
  \log_+\frac{1}{\cM(T)} \,\le\, \frac{C_7\cM(T)}{|\alpha|^{1/6}} 
  \,\le\, \frac{\cM(T)}{4}\,,
\]
and inequality \eqref{nonlinest2} then shows that $\cM(T) <  C_4 
\|\tilde w(0)\|_X + \cM(T)/2$, which in turn implies that $\cM(T) < 
2C_4\|\tilde w_0\|_X$. So we conclude that \eqref{Mcontract} holds 
for any $T > 0$, namely
\begin{equation}\label{Mbound}
  \sup_{\tau \ge 0} \Bigl(\|\tilde w_r(\tau)\|_X + e^{\tau/\tau_0} 
  \|\tilde w_\perp(\tau)\|_X\Bigr) \,\le\, 2 C_4 \|\tilde w_0\|_X\,.
\end{equation}
In view of the definition \eqref{mutaudef} of $\tau_0$, this proves 
in particular \eqref{wperpdecay}. To establish \eqref{wrdecay} 
for $\tau \ge 1$, we observe that the integral equation \eqref{wrpint} 
satisfied by the radially symmetric component $\tilde w_r$ can be written 
in the form
\begin{equation}\label{wrint2}
  \tilde w_r(\tau) \,=\, e^{\tau\cL}\tilde w_r(0) - \int_0^\tau 
  e^{-\frac12(\tau-s)}\div\Bigl(e^{(\tau-s)\cL} Z[\tilde w(s),\tilde v(s)]\Bigr)
  \dd s\,,
\end{equation}
see \eqref{wint} below. As $\tilde w_0 \in X_1$, we have 
$\|e^{\tau\cL}\tilde w_r(0)\|_X \le e^{-\tau}\|\tilde w_r(0)\|_X$ 
for all $\tau \ge 0$. Moreover, each component of the velocity 
field $Z[\tilde w,\tilde v]$ belongs to the subspace $X_0$, hence 
using the semigroup estimates in \cite[Appendix~A]{GW1} and 
proceeding exactly as in the proof of \eqref{bilin} below, we 
obtain the bound
\begin{equation}\label{wrintest}
  \Bigl\|\div\Bigl(e^{(\tau-s)\cL} Z[\tilde w(s),\tilde v(s)]\Bigr)
  \Bigr\|_X \,\le\, C\,\frac{e^{-\frac12(\tau-s)}}{a(\tau-s)^{3/4}}
  \,\|\tilde w_\perp(s)\|_X^2\,,
\end{equation}
for $0 < s < \tau$, where $a(\tau) = 1-e^{-\tau}$. Estimating the
right-hand side of \eqref{wrint2} for $\tau \ge 1$ with the help 
of \eqref{wrintest} and \eqref{Mbound}, we easily find
\begin{align*}
  \|\tilde w_r(\tau)\|_X \,&\le\, C e^{-\tau}\|\tilde w_0\|_X 
  + C \|\tilde w_0\|_X^2 \int_0^\tau \frac{e^{-(\tau-s)}}{a(\tau-s)^{3/4}}
  \,e^{-2s/\tau_0}\dd s \\
  \,&\le\, C e^{-\tau}\Bigl(\|\tilde w_0\|_X + \tau_0 
  \|\tilde w_0\|_X^2\Bigr) \,\le\, C e^{-\tau} \|\tilde w_0\|_X\,,
\end{align*}
where in the last inequality we used the fact that $\tau_0 
\|\tilde w_0\|_X$ is uniformly bounded if $\tau_0$ is defined by 
\eqref{mutaudef} and $\tilde w_0$ satisfies \eqref{initbound}. 
Together with \eqref{Mbound}, this concludes the proof of 
\eqref{wrdecay}. \QED

\section{Appendix}
\label{sec5}

In this final section, we collect for easy reference a few basic results 
concerning the rescaled vorticity equation \eqref{weq} and the 
two-dimensional Biot-Savart law \eqref{BS}. In particular 
we give a short proof of Propositions~\ref{prop:globex} and
\ref{prop:locstab}. 

\subsection{Well-posedness and a priori estimates}\label{subsec5.1}

We first recall that the Cauchy problem for the rescaled vorticity 
equation \eqref{weq} is globally well-posed in the weighted space 
$X = L^2(\R^2,G^{-1}\dd\xi)$ equipped with the scalar product 
\eqref{Xscalar}. Indeed, the integral equation associated with 
\eqref{weq} reads
\begin{equation}\label{wint}
 \begin{aligned}
  w(\tau) \,&=\, e^{\tau\cL} w_0 - \int_0^\tau e^{(\tau-s)\cL} \div\bigl(v(s)
  w(s)\bigr)\dd s \\
  \,&=\, e^{\tau\cL} w_0 - \int_0^\tau e^{-\frac12(\tau-s)} \div\bigl(e^{(\tau-s)\cL} 
  v(s)w(s)\bigr)\dd s\,, \qquad \tau \ge 0\,,
  \end{aligned}
\end{equation}
where in the second inequality we used the identity $\partial_i\,e^{\tau\cL} 
= e^{\tau(\cL+\frac12)}\partial_i$ for $i = 1,2$, which itself follows from its
infinitesimal version $\partial_i\cL = (\cL + \frac12) \partial_i$. In 
\eqref{wint} it is understood that the velocity $v(s)$ is reconstructed 
from the vorticity $w(s)$ via the two-dimensional Biot-Savart law 
\eqref{BS}. 

There is an explicit expression for the semigroup $e^{\tau\cL}$ generated by 
$\cL$, which can be found for instance in \cite[Appendix~A]{GW1}\:
\begin{equation}\label{mehler1}
  \bigl(e^{\tau\cL}w\bigr)(\xi) \,=\, \frac{1}{4\pi a(\tau)}
  \int_{\R^2} e^{-\frac{1}{4 a(\tau)}|\xi-\eta e^{-\tau/2}|^2}\,w(\eta)\dd \eta\,, 
  \qquad \xi \in \R^2\,,~\tau > 0\,,
\end{equation}
where $a(\tau) = 1-e^{-\tau}$. Since we work in the space $X$ with 
Gaussian weight $G^{-1/2}(\xi) = \sqrt{4\pi}\,e^{|\xi|^2/8}$, it is more 
convenient here to use Mehler's formula
\begin{equation}\label{mehler2}
 \begin{aligned}
  \bigl(e^{\tau L} f\bigr)(\xi) \,&=\, \frac{1}{4\pi a(\tau)}
  \int_{\R^2} e^{|\xi|^2/8}\,e^{-\frac{1}{4 a(\tau)}|\xi-\eta e^{-\tau/2}|^2}
  \,e^{-|\eta|^2/8}\,f(\eta)\dd \eta\,, \\ 
  \,&=\, \frac{1}{4\pi a(\tau)}\int_{\R^2} e^{-\frac{1}{8 a(\tau)}\bigl(
  |\xi-\eta e^{-\tau/2}|^2 + |\eta-\xi e^{-\tau/2}|^2\bigr)} \,f(\eta)\dd \eta\,, 
 \end{aligned}
\end{equation}
which defines the semigroup generated by the operator $L = G^{-1/2}\cL 
\,G^{1/2}$, see \eqref{Lconj}. A direct calculation based on \eqref{mehler2} 
gives the $L^p$--$L^q$ estimates $\|e^{\tau L} f\|_{L^q(\R^2)} \le C 
a(\tau)^{\frac1q-\frac1p}\|f\|_{L^p(\R^2)}$ for $1 \le p \le q \le \infty$. 
Returning to the original operator $\cL$, we conclude that
\begin{equation}\label{LpLq}
   \|G^{-1/2}\,e^{\tau \cL} w\|_{L^q(\R^2)} \,\le\, \frac{C}{a(\tau)^{\frac1p-\frac1q}}\,
  \|G^{-1/2}\,w\|_{L^p(\R^2)}\,, \qquad \tau > 0\,, 
\end{equation}
where  $a(\tau) = 1-e^{-\tau}$. Similarly, we have the corresponding estimate 
for the first derivatives\:
\begin{equation}\label{LpLqder}
   \|G^{-1/2}\nabla\,e^{\tau \cL} w\|_{L^q(\R^2)} \,\le\, \frac{C}{a(\tau)^{\frac1p
   -\frac1q + \frac12}}\,\|G^{-1/2}\,w\|_{L^p(\R^2)}\,, \qquad \tau > 0\,,
\end{equation}
see \cite[Proposition~2.1]{GM1}. 

To establish local well-posedness for Eq.~\eqref{weq} in the space $X$ 
it is sufficient to prove that the bilinear operator
\[
  B[w_1,w_2](\tau) \,=\, \int_0^\tau e^{-\frac12(\tau-s)} \div\bigl(e^{(\tau-s)\cL} 
  v_1(s)w_2(s)\bigr)\dd s\,, \quad \hbox{where}~
  v_1 = K_{BS}*w_1\,,
\]
is continuous in the space $C^0([0,T],X)$ for any $T > 0$, and has 
a small norm if $T \ll 1$. Proceeding as in \cite[Lemma~3.1]{GW1}, 
we deduce from \eqref{LpLqder} with $q = 2$ and $p = 4/3$ that
\[
   \|B[w_1,w_2](\tau)\|_X \,\le\, \int_0^\tau \frac{C}{a(\tau-s)^{3/4}}\,
   \|G^{-1/2} v_1(s) w_2(s)\|_{L^{4/3}}\dd s\,.
\]
Next, using H\"older's inequality and Lemma~\ref{lem:BS1} below
(with $p = 4/3$, $q = 4$), we obtain
\begin{equation}\label{Gv1w2}
  \|G^{-1/2} v_1 w_2\|_{L^{4/3}} \,\le\, \|v_1\|_{L^4} \|G^{-1/2} w_2\|_{L^2}
  \,\le\, C \|w_1\|_{L^{4/3}} \|w_2\|_X \,\le\, C \|w_1\|_X  \|w_2\|_X\,.
\end{equation}
Thus we conclude that
\begin{equation}\label{bilin}
  \sup_{0 \le \tau \le T} \|B[w_1,w_2](\tau)\|_X \,\le\, \Bigl(\int_0^T 
  \frac{C}{a(s)^{3/4}} \dd s\Bigr)\,\Bigl(\,\sup_{0 \le s \le T}
  \|w_1(s)\|_X\Bigr)\,\Bigl(\,\sup_{0 \le s \le T} \|w_2(s)\|_X\Bigr)\,,
\end{equation}
and this bilinear estimate implies local well-posedness in $X$ by a standard 
fixed point argument. 

To prove that all solutions of \eqref{weq} in $X$ are global we 
need to show that the norm $\|w(\tau)\|_X$ cannot blow-up in finite
time. Sharp a priori estimates on the solutions of the original 
vorticity equation \eqref{omeq} have been obtained by Carlen and 
Loss in \cite{CL}, and can be translated into useful bounds for
the rescaled equation \eqref{weq}. In particular, it follows 
from \cite[Theorem~3]{CL} that any solution of \eqref{weq} with 
initial data $w_0 \in L^1(\R^2)$ satisfies the pointwise estimate
\begin{equation}\label{CL1}
  |w(\xi,\tau)| \,\le\, \frac{C_\beta(R)}{4\pi a(\tau)}
  \int_{\R^2} e^{-\frac{\beta}{4 a(\tau)}|\xi-\eta e^{-\tau/2}|^2}\,
  |w_0(\eta)| \dd\eta\,, \qquad \xi \in \R^2\,,~\tau > 0\,,
\end{equation}
for any $\beta \in (0,1)$, where $R = \|w_0\|_{L^1(\R^2)}$ and 
$C_\beta(R) = \exp(\frac{\beta}{1-\beta} \frac{R^2}{2\pi^2})$, 
see also \cite[Section~2]{GR}. If we assume that $w_0 \in X$ 
and $\beta > 1/2$, a direct calculation based on \eqref{CL1} 
shows that
\begin{equation}\label{CL2}
  \int_{\R^2} e^{|\xi|^2/4} |w(\xi,\tau)|^2\dd \xi \,\le\, 
  \frac{4C_\beta^2}{(2\beta {-} 1 {+} e^{-\tau})(1 {+} (2\beta {-} 
  1)e^{-\tau})}\int_{\R^2} e^{|\eta|^2/4} |w_0(\eta)|^2\dd \eta\,,
\end{equation}
for all $\tau > 0$. In particular, we have $\|w(\tau)\|_X \le 
2 C_\beta (2\beta{-}1)^{-1/2} \|w_0\|_X$ for all $\tau \ge 0$, and 
this implies that all solutions of \eqref{weq} in $X$ are 
global and uniformly bounded for positive times. 

Finally, we prove that all solutions of \eqref{weq} in $X$ 
converge to $\alpha G$ as $\tau \to \infty$, where
\[
  \alpha \,=\, \int_{\R^2} w(\xi,\tau)\dd \xi \,=\, 
  \int_{\R^2} w_0(\xi)\dd \xi\,.
\]
Indeed, we decompose $w(\tau) = \alpha G + \tilde w(\tau)$ for 
all $\tau \ge 0$, and we consider the equation \eqref{tildeweq} 
satisfied by the perturbation $\tilde w \in X_0$. Using 
the skew-symmetry of the operator $\Lambda$ and integrating 
by parts, we easily obtain the energy estimate
\begin{equation}\label{wtildest1}
\begin{aligned}
  \frac12\frac{\D}{\D \tau} \|\tilde w(\tau)\|_X^2 \,&=\, 
  \langle \tilde w(\tau), \cL \tilde w(\tau) \rangle_X - \int_{\R^2} 
  G^{-1} \tilde w(\xi,\tau) \tilde v(\xi,\tau)\cdot\nabla\tilde 
  w(\xi,\tau)\dd\xi \\ \,&=\, 
  \langle \tilde w(\tau), \cL \tilde w(\tau) \rangle_X +\frac14 \int_{\R^2} 
  G^{-1} (\xi\cdot \tilde v(\xi,\tau)) \tilde w(\xi,\tau)^2 \dd\xi\,.
\end{aligned}
\end{equation}
The first term in the right-hand side is estimated using the 
following lemma. 

\begin{lem}\label{lem:E}
Given $w \in D(\cL) \subset X$, define $E[w] = - \langle w, 
\cL w\rangle_X \ge 0$.\\[1mm]
i) If $w \in X_0$, then $E[w] \ge \frac12 \|w\|_X^2$ and
\begin{equation}\label{E1}
  E[w] \,\ge\, \frac16 \|\nabla w\|_X^2 + \frac{1}{96} \|\xi w\|_X^2 + 
  \frac{1}{12} \|w\|_X^2\,.
\end{equation}
ii) If $w \in X_1$, then $E[w] \ge \|w\|_X^2$ and
\begin{equation}\label{E2}
  E[w] \,\ge\, \frac14 \|\nabla w\|_X^2 + \frac{1}{64} \|\xi w\|_X^2 + 
  \frac18 \|w\|_X^2\,.
\end{equation}
\end{lem}

\begin{proof}
Using definition \eqref{cLdef} and integrating by parts, we obtain 
\begin{equation}\label{E3}
  E[w] \,=\, -\int_{\R^2} G^{-1} w (\cL w)\dd \xi 
  \,=\, \int_{\R^2} G^{-1} |\nabla w|^2 \dd \xi - \int_{\R^2} G^{-1} 
  |w|^2 \dd \xi\,. 
\end{equation}
Moreover it is easy to verify that
\[
  \int_{\R^2} G^{-1} |\nabla w|^2 \dd \xi \,=\, \int_{\R^2} |\nabla 
  (G^{1/2} w)|^2 \dd \xi + \frac{1}{16}\int_{\R^2} G^{-1} |\xi|^2 |w|^2
  \dd \xi + \frac{1}{2}\int_{\R^2} G^{-1} |w|^2 \dd \xi\,, 
\]
hence
\begin{equation}\label{E4}
  E[w] \,\ge\, \frac{1}{16}\int_{\R^2} G^{-1} |\xi|^2 |w|^2
  \dd \xi - \frac{1}{2}\int_{\R^2} G^{-1} |w|^2 \dd \xi\,.
\end{equation}
Finally, the dissipative properties of $\cL$ recalled in 
Section~\ref{subsec2.1} imply that $E[w] \ge 0$, and 
\begin{equation}\label{E5}
  E[w] \,\ge\, \frac12 \int_{\R^2} G^{-1} |w|^2 \dd \xi
  \quad\hbox{if } w \in X_0\,, \qquad 
  E[w] \,\ge\, \int_{\R^2} G^{-1} |w|^2 \dd \xi
  \quad\hbox{if } w \in X_1\,.
\end{equation}
Taking a convex combination of \eqref{E3}, \eqref{E4}, \eqref{E5}
with coefficients $1/6$, $1/6$, $2/3$ if $w \in X_0$, and 
with coefficients $1/4$, $1/4$, $1/2$ if $w \in X_1$, we 
obtain estimates \eqref{E1}, \eqref{E2}, respectively. 
\end{proof}

We now return to the analysis of \eqref{wtildest1}. Using 
H\"older's inequality and Lemma~\ref{lem:E}, the cubic term
can be estimated as follows\:
\begin{equation}\label{wtildest2}
  \Bigl|\int_{\R^2} G^{-1} (\xi\cdot \tilde v) \tilde w^2\dd\xi \Bigr| 
  \,\le\, \|G^{-1/2} |\xi| \tilde w\|_{L^2} \|G^{-1/2}\tilde w\|_{L^4} 
   \|\tilde v\|_{L^4} \,\le\, C E|\tilde w]\,\|\tilde v\|_{L^4}\,,
\end{equation}
because $\|G^{-1/2} |\xi| \tilde w\|_{L^2} = \||\xi| \tilde w\|_X 
\le C E[\tilde w]^{1/2}$ and
\[
  \|G^{-1/2}\tilde w\|_{L^4} \,\le\, C \|G^{-1/2}\tilde w\|_{L^2}^{1/2}\,
  \|\nabla (G^{-1/2}\tilde w)\|_{L^2}^{1/2} \,\le\, C E[\tilde w]^{1/2}\,.
\]
As $\tilde w \in X_0$, we obtain from \eqref{wtildest1}, \eqref{wtildest2} 
and Lemma~\ref{lem:E}\:
\begin{equation}\label{wtildest3}
  \frac{\D}{\D \tau} \|\tilde w(\tau)\|_X^2 \,\le\,  
  - 2 E[\tilde w(\tau)] \bigl(1 - C_8 \|\tilde v(\tau)\|_{L^4}\bigr)\,,
\end{equation}
for some positive constant $C_8 > 0$. Now, since our Gaussian space $X$
is included in the polynomially weighted spaces $L^2(m)$ considered in 
\cite{GW2}, we know from Lemma~\ref{lem:BS1} and 
\cite[Proposition~1.5]{GW2} that $\|\tilde v(\tau)\|_{L^4} \le C 
\|\tilde w(\tau)\|_{L^{4/3}} \to 0$ as $\tau \to +\infty$. In particular, 
we have $C_8 \|\tilde v(\tau)\|_{L^4} \le 1/2$ for large times, in which 
case we deduce from \eqref{wtildest3} and Lemma~\ref{lem:E} that
\[
  \frac{\D}{\D \tau} \|\tilde w(\tau)\|_X^2 \,\le\,  
  - \|\tilde w(\tau)\|_X^2 \bigl(1 - C_8 \|\tilde v(\tau)
  \|_{L^4}\bigr) \,\le\, -\frac12 \|\tilde w(\tau)\|_X^2\,. 
\]
This differential inequality implies that $\|\tilde w(\tau)\|_X = 
\|w(\tau) - \alpha G\|_X \to 0$ as $\tau \to \infty$, which concludes 
the proof of Proposition~\ref{prop:globex}. 

On the other hand, if we restrict ourselves to solutions in a 
neighborhood of $\alpha G$, we can assume from the beginning that
$C_8\|\tilde v(\tau)\|_{L^4} \le C_9\|\tilde w(\tau)\|_X \le 1/2$, in 
which case \eqref{wtildest3} implies
\[
  \frac{\D}{\D \tau} \|\tilde w(\tau)\|_X^2 \,\le\,  - 2 E[\tilde w(\tau)]  
  \bigl(1 - C_9 \|\tilde w(\tau)\|_X\bigr) \,\le\,  - \|\tilde w(\tau)\|_X^2 
  \bigl(1 - C_9 \|\tilde w(\tau)\|_X\bigr)\,,
\]
hence
\[
  \frac{\|\tilde w(\tau)\|_X}{1 - C_9 \|\tilde w(\tau)\|_X} \,\le\, 
  \frac{\|\tilde w(0)\|_X}{1 - C_9 \|\tilde w(0)\|_X}\,e^{-\tau/2}\,,
  \qquad \tau \ge 0\,. 
\]
We easily deduce that $\|\tilde w(\tau)\|_X \,\le\, \min(1,2\,e^{-\tau/2}) 
\|\tilde w(0)\|_X$ for all $\tau \ge 0$, and this concludes the 
proof of Proposition~\ref{prop:locstab}.

\subsection{The two-dimensional Biot-Savart law}\label{subsec5.2}

Let $u = K_{BS}*\omega$ be the velocity field associated with the 
vorticity distribution $\omega$ via the Biot-Savart law \eqref{BS}. 
We first recall the following classical estimates.  

\begin{lem}\label{lem:BS1}\quad\\[1mm]
1) Assume that $1 < p < 2 < q < \infty$ and $\frac{1}{q} = \frac{1}{p} - 
\frac{1}{2}$. If $\omega \in L^p(\R^2)$, then $u \in L^q(\R^2)$ and
\begin{equation}\label{BS1}
  \|u\|_{L^q(\R^2)} \,\le\, C_q \|\omega\|_{L^p(\R^2)}~.
\end{equation}
2) Assume that $1 \le p < 2 < q \le \infty$. If $\omega \in L^p(\R^2) 
\cap L^q(\R^2)$, then $u \in L^\infty(\R^2)$ and
\begin{equation}\label{BS2}
  \|u\|_{L^\infty(\R^2)} \,\le\, C_{p,q} \|\omega\|_{L^p(\R^2)}^\theta
  \|\omega\|_{L^q(\R^2)}^{1-\theta}\,,
\end{equation}
where $\theta \in (0,1)$ satisfies $\frac{\theta}{p} + \frac{1-\theta}{q} 
= \frac12$. 
\end{lem}

\begin{proof}
The bound \eqref{BS1} is a direct consequence of the Hardy-Littlewood-Sobolev
inequality, see e.g. \cite{LL}. The optimal constant $C_q$ in 
\eqref{BS1} is not explicitly known, but one can show that $C_q = 
\cO(q^{1/2})$ as $q \to +\infty$. Estimate \eqref{BS2} is a 
consequence of inequality \eqref{BSfirst} below, where the constant
$C_{p,q}$ is explicitly given. 
\end{proof}

Next we establish a new estimate, somewhat in the spirit of the celebrated 
Br\'ezis-Gallou\"et inequality \cite{BG}, which strengthens \eqref{BS2}
and shows that the $L^\infty$ norm of the velocity field can be controlled 
by the $L^2$ norm of the vorticity up to logarithmic correction. 

\begin{lem}\label{lem:BS2} 
Assume that $1 \le p < 2 < q \le \infty$. If $\omega \in L^p(\R^2) 
\cap L^q(\R^2)$, then $u \in L^\infty(\R^2)$ and
\begin{equation}\label{BS3}
  \|u\|_{L^\infty} \,\le\, C \|\omega\|_{L^2} \Bigl(1 + \log\frac{
  \|\omega\|_{L^p}^\theta\|\omega\|_{L^q}^{1-\theta}}{
  \|\omega\|_{L^2}}\Bigl)^{1/2}\,,
\end{equation}
where $C$ depends only on $p,q$, and $\theta \in (0,1)$ satisfies 
$\frac{\theta}{p} + \frac{1-\theta}{q} = \frac12$. 
\end{lem}

\begin{proof}
For any $x \in \R^2$ and any $R > 0$, we have
\[
  |u(x)| \,\le\, \frac{1}{2\pi}\int_{|y|<R} \frac{1}{|y|}\,
  |\omega(x-y)|\dd y + \frac{1}{2\pi}\int_{|y|\ge R} \frac{1}{|y|}\,
  |\omega(x-y)|\dd y\,.
\]
Fix $1 < p < 2 < q < \infty$, and let $p',q'$ be the conjugate 
exponents to $p,q$, respectively. Applying H\"older's inequality, 
we obtain
\begin{align*}
  \|u\|_{L^\infty} \,&\le\, \frac{1}{2\pi}\biggl(\int_{|y|<R} 
  \frac{1}{|y|^{q'}}\dd y\biggr)^{1/q'} \|\omega\|_{L^q} + 
  \frac{1}{2\pi}\biggl(\int_{|y|\ge R} \frac{1}{|y|^{p'}}\dd y
  \biggr)^{1/p'} \|\omega\|_{L^p} \\
  \,&=\, C\Bigl(\frac{1}{2{-}q'}\Bigr)^{1/q'} R^{1-2/q} \,\|\omega\|_{L^q} + 
  C\Bigl(\frac{1}{p'{-}2}\Bigr)^{1/p'} \frac{1}{R^{-1+2/p}} \,\|\omega\|_{L^p} \\
  \,&=\, C \Bigl(\frac{q{-}1}{q{-}2}\Bigr)^{1-1/q} R^{1-2/q} \,\|\omega\|_{L^q} + 
  C \Bigl(\frac{p{-}1}{2{-}p}\Bigr)^{1-1/p} \frac{1}{R^{-1+2/p}} \,\|\omega\|_{L^p}\,. 
\end{align*}
If we optimize the right-hand side over all possible values of $R > 0$, 
we arrive at
\begin{equation}\label{BSfirst}
  \|u\|_{L^\infty} \,\le\, C\,
  \Bigl(\frac{p{-}1}{2{-}p}\Bigr)^{\frac{(q-2)(p-1)}{2(q-p)}}\,
  \Bigl(\frac{q{-}1}{q{-}2}\Bigr)^{\frac{(q-1)(2-p)}{2(q-p)}}\,
  \,\|\omega\|_{L^p}^{\frac{p(q-2)}{2(q-p)}}\, 
  \,\|\omega\|_{L^q}^{\frac{q(2-p)}{2(q-p)}}\,.
\end{equation}
Remark that the right-hand side of \eqref{BSfirst} has a finite 
limit as $p \to 1$ or $q \to \infty$, so that a similar estimate
holds in these limiting cases too. In particular, we have 
$\|u\|_{L^\infty} \le C \|\omega\|_{L^1}^{1/2}\|\omega\|_{L^\infty}^{1/2}$. 

We next assume that $p = 2-\epsilon$ and $q = 2+\epsilon$ for some 
$\epsilon \in (0,1/2]$. With this choice, we have
\[
  \Bigl(\frac{q{-}1}{q{-}2}\Bigr)^{\frac{(q-1)(2-p)}{2(q-p)}} \,=\, 
  \Bigl(\frac{1+\epsilon}{\epsilon}\Bigr)^{\frac{1+\epsilon}{4}} \,\le\, 
  \frac{C}{\epsilon^{1/4}}\,, \qquad
  \Bigl(\frac{p{-}1}{2{-}p}\Bigr)^{\frac{(q-2)(p-1)}{2(q-p)}} \,=\,
  \Bigl(\frac{1-\epsilon}{\epsilon}\Bigr)^{\frac{1-\epsilon}{4}} \,\le\, 
  \frac{C}{\epsilon^{1/4}}\,,
\]
so that \eqref{BSfirst} reduces to
\begin{equation}\label{BSsecond}
  \|u\|_{L^\infty} \,\le\, \frac{C}{\epsilon^{1/2}}\, 
  \|\omega\|_{L^{2-\epsilon}}^{\frac{2-\epsilon}{4}}\,
  \|\omega\|_{L^{2+\epsilon}}^{\frac{2+\epsilon}{4}}\,.
\end{equation}

In the rest of the proof, we show how estimate \eqref{BS3} can be 
deduced from \eqref{BSsecond}. Let again $1 \le p < 2 < q \le \infty$. 
If $0 < \epsilon \le \epsilon_0 := \min(2-p,q-2)$, we can interpolate
\[
  \|\omega\|_{L^{2-\epsilon}} \,\le\, \|\omega\|_{L^p}^\alpha 
  \,\|\omega\|_{L^2}^{1-\alpha}\,, \qquad
  \|\omega\|_{L^{2+\epsilon}} \,\le\, \|\omega\|_{L^q}^\beta 
  \,\|\omega\|_{L^2}^{1-\beta}\,,
\]
where $\frac{\alpha}{p} + \frac{1-\alpha}{2} = \frac{1}{2-\epsilon}$
and $\frac{\beta}{q} + \frac{1-\beta}{2} = \frac{1}{2+\epsilon}$. 
Substituting into \eqref{BSsecond}, we obtain after straightforward 
calculations
\begin{equation}\label{BSthird}
  \|u\|_{L^\infty} \,\le\, \frac{C}{\epsilon^{1/2}}\, \|\omega\|_{L^2} 
  \,A^{\epsilon \gamma/2}\,, \qquad A \,=\, \frac{\|\omega\|_{L^p}^\theta
  \|\omega\|_{L^q}^{1-\theta}}{\|\omega\|_{L^2}} \,\ge\, 1\,,
\end{equation}
where $\theta$ is as in the statement and $\gamma = \frac{1}{2-p} + 
\frac{1}{q-2}$. It remains to optimize the choice of $\epsilon \in 
(0,\epsilon_0]$. If $A \le \exp((\epsilon_0 \gamma)^{-1})$, we 
take $\epsilon = \epsilon_0$ and \eqref{BSthird} implies
\[
  \|u\|_{L^\infty} \,\le\, C \epsilon_0^{-1/2} \,A^{\epsilon_0 \gamma/2}
  \,\|\omega\|_{L^2} \,\le\, C \epsilon_0^{-1/2} e^{1/2} \,\|\omega\|_{L^2}\,.
\]
If $A > \exp((\epsilon_0 \gamma)^{-1})$, we take $\epsilon = (\gamma 
\log(A))^{-1} < \epsilon_0$, and we obtain
\[
  \|u\|_{L^\infty} \,\le\, C (\gamma \log(A))^{1/2} A^{\frac{1}{2\log A}} 
  \,\|\omega\|_{L^2} \,\le\, C (\log(A))^{1/2} \|\omega\|_{L^2}\,.
\]
We conclude that estimate \eqref{BS3} holds in all cases. 
\end{proof}

We conclude this section with two additional results in the 
spirit of Lemma~\ref{BS2}, which are tailored for our purposes 
in Section~\ref{sec4}. 

\begin{lem}\label{lem:BS3} 
Assume that $\omega \in L^1(\R^2) \cap L^3(\R^2)$, and let
$u = K_{BS}*\omega$. Then 
\begin{equation}\label{BS4}
  \|u\|_{L^\infty} \,\le\, C \|\omega\|_{L^1 \cap L^2} 
  \Bigl(1 + \log_+ \frac{\|\omega\|_{L^3}}{\|\omega\|_{L^1 
  \cap L^2}}\Bigl)^{1/2}\,,
\end{equation}
where $C > 0$ is a universal constant and $\|\omega\|_{L^1 \cap L^2} = 
\max(\|\omega\|_{L^1},\|\omega\|_{L^2})$. 
\end{lem}

\begin{proof}
We use inequality \eqref{BSthird} with $p = 1$ and $q = 3$, so that 
$\gamma = 2$. This gives
\[
  \|u\|_{L^\infty} \,\le\, \frac{C}{\epsilon^{1/2}}\, 
  \|\omega\|_{L^1}^{\epsilon/4} \,\|\omega\|_{L^2}^{1-\epsilon}\,
  \|\omega\|_{L^3}^{3\epsilon/4} \,\le\, \frac{C}{\epsilon^{1/2}}\, 
  \|\omega\|_{L^1\cap L^2}^{1-3\epsilon/4}\,\|\omega\|_{L^3}^{3\epsilon/4}\,. 
\]
Optimizing the choice of $\epsilon > 0$ as in the proof of 
Lemma~\ref{lem:BS2}, we arrive at \eqref{BS4}. 
\end{proof}

\begin{lem}\label{lem:BS4} 
Assume that $\omega_1 \in L^2(\R^2) \cap L^3(\R^2)$, $\omega_2 
\in L^1(\R^2) \cap L^2(\R^2)$, and let $u_2 = K_{BS} * \omega_2$. 
Then 
\begin{equation}\label{BS5}
  \|\omega_1 u_2 \|_{L^2} \,\le\, C \|\omega_1\|_{L^2}
  \|\omega_2\|_{L^1 \cap L^2} \Bigl(1 + \log_+ \frac{
  \|\omega_1\|_{L^3}}{\|\omega_1\|_{L^2}}\Bigl)^{1/2}\,.
\end{equation}
\end{lem}

\begin{proof}
Applying H\"older's inequality and Lemma~\ref{lem:BS1}, we obtain 
for any $q \ge 6$\:
\[
  \|\omega_1 u_2 \|_{L^2} \,\le\, \|\omega_1\|_{L^r}\|u_2\|_{L^q}
  \,\le\, C q^{1/2} \|\omega_1\|_{L^r} \|\omega_2\|_{L^p}\,,
\]
where $\frac1r + \frac1q = \frac12$ and $\frac1q = \frac1p - \frac12$. 
By interpolation, we have $\|\omega_2\|_{L^p} \le \|\omega_2\|_{L^1\cap L^2}$
and
\[
   \|\omega_1\|_{L^r} \,\le\, \|\omega_1\|_{L^2}^{-2+6/r} 
   \|\omega_1\|_{L^3}^{3-6/r} \,=\, C \|\omega_1\|_{L^2}^{1-6/q} 
   \|\omega_1\|_{L^3}^{6/q}\,,
\]
hence
\[
   \|\omega_1 u_2 \|_{L^2} \,\le\, C q^{1/2} \|\omega_1\|_{L^2}^{1-6/q} 
   \,\|\omega_1\|_{L^3}^{6/q} \,\|\omega_2\|_{L^1\cap L^2}\,.
\]
Optimizing the choice of $q \in [6,\infty)$ as in the proof of 
Lemma~\ref{lem:BS2}, we arrive at \eqref{BS5}. 
\end{proof}


\end{document}